\documentclass{amsart}
\usepackage{va}

\usepackage{ifthen}
\newboolean{springer}\setboolean{springer}{false}

\newboolean{longversion}\setboolean{longversion}{false}

\usepackage{longtable}

\definecolor{amaranth}{rgb}{0.9, 0.17, 0.31}
\usepackage{hyperref}
\hypersetup{
  colorlinks=true,
  citecolor=amaranth,
  linkcolor=black,
  filecolor=magenta,      
  urlcolor=blue,
}
 
\usepackage{tikz, tikz-cd} 
\usepackage{orcidlink}
\usepackage{bm}

\ifthenelse{\boolean{longversion}}{%
  \newcommand\presectionskip{\medskip}
}{%
  \newcommand\presectionskip{\medskip}
}

\numberwithin{equation}{section}

\let\oldtocsubsection=\tocsubsection
\renewcommand{\tocsubsection}[2]{\hspace{1em}\oldtocsubsection{#1}{#2}}

\newcommand\nuu{\nu}

\newcommand\bone{\bm{1}}
\newcommand\chr{\operatorname{char}}
\newcommand\Def{\operatorname{Def}}
\newcommand\Sym{\operatorname{Sym}}

\newcommand\ovval{\overline{\val}\, }

\newcommand\cym{\cE} 
\newcommand\cymm{E} 

\newcommand\pd{\partial}
\newcommand\newton{\Gamma_+}

\newcommand\mymod[1]{\,({\rm mod}\ #1)}
\newcommand\lct{\operatorname{lct}}

\newcommand\orb{_{\mathrm orb}}

\newcommand\dotcym{\dot\cym}
\newcommand\dotcX{\dot\cX}
\newcommand\amb{\bP}
\newcommand\amba{\bP_A}

\title[A tower of complete moduli spaces of Calabi-Yau $n$-folds]%
{A tower of complete moduli spaces\\of Calabi-Yau $n$-folds}
\date{March 21, 2026}
\author{Valery Alexeev\,\orcidlink{0000-0001-6730-0161}}
\email{valery@uga.edu}
\address{Department of Mathematics, University of Georgia, Athens GA
  30602, USA}

\begin{document} 
\begin{abstract}
  We construct a sequence of complete moduli spaces
  $$\cymm_0 \subset \cymm_1 \subset \cymm_2 \subset \dots \cymm_n \subset
  \dots,$$ each of which is isomorphic to a weighted projective space. These spaces
  parameterize certain $n$-dimensional Calabi-Yau varieties
  associated with the Sylvester sequence $2,3,7,43,\dots$.  They
  generalize the moduli space of elliptic curves $\oM_{1,1}=\bP(4,6)$
  and Brieskorn's family over
  $\oF^{\rm BB}_{U\oplus E_8} = \bP(4,10,\dotsc, 42)$, the Baily-Borel
  compactification of the moduli space of $U\oplus E_8$-polarized K3
  surfaces.
  We also study fibrations in such Calabi-Yau varieties, extending to
  higher dimensions the theory of elliptic surfaces.
\end{abstract}
\maketitle

\setcounter{tocdepth}{1}
\tableofcontents

\section{Introduction}
\label{sec:intro}

Let $\{s_0, s_1, s_2, s_3, \dotsc \} = \{2, 3, 7, 43, \dotsc \}$ be
the Sylvester sequence defined recursively by the relation
$s_n=\prod_{k=0}^{n-1} s_k + 1$. Denote $d=\prod_{k=0}^n s_k$ and
$a_k=d/s_k$. Consider the Brieskorn-Pham hypersurface singularity
\begin{equation}\label{eq:brieskorn-pham} 
  ( f_0 =0 ) \subset \bC^{n+1}, \qquad
  f_0 = x_0^{s_0} + x_1^{s_1} + \dotsb + x_{n}^{s_{n}}.
\end{equation}
It is weighted homogeneous (also called quasihomogeneous) of degree $d$ if we
set $\deg x_k = a_k$. Consider its semiuniversal unfolding 
\begin{equation}\label{eq:f_t}
  f_t = f_0 + \sum_I t_I x^I, \quad
  t_I x^I = t_{i_0\dotsc i_n} x_0^{i_0}\dotsb x_n^{i_n}
  \quad\text{with } 0\le i_k \le s_k-2
\end{equation}
over $\bA^\mu$, where $\mu=\prod_{k=0}^n (s_k-1)$ is the Milnor number
(the same as the Tjurina number) of $f_0$.
For an $(n+1)$-tuple $I= (i_0,\dotsc, i_n)$, define
the \emph{degree difference} of $x^I$:
\begin{equation}
  \label{eq:degree-diff}
  \delta_I = d - \sum_{k=0}^n a_{k} i_k =
  d \Big( 1 - \sum_{k=0}^n \frac{i_k}{ s_k} \Big) \in \bZ.
\end{equation}
Because the $s_k$ are pairwise coprime, there are no monomials $t_Ix^I$ with
$\delta_I=0$. If all the monomials $t_Ix^I$ with $t_I\ne 0$ in $f_t$
have weighted degree $>d$, i.e.\ $\delta_I<0$, then $(f_t=0)$ is a
semi-quasihomogeneous singularity with the same Milnor number
$\mu(f_t)=\mu$. These monomials define the $\mu$-constant stratum
$W^+$ in the semiuniversal deformation of $f_0$, and $\dim W^+$ equals
the modality of $f_0$, see
\cite{arnold1974normal-forms,varchenko1982lower-bound}.  The space
$W^+$ is frequently referred to as the equisingular deformation space
of $f_0$. The hypersurfaces $(f_t=0)$ in this space are homeomorphic
to $(f_0=0)$: for $n\ge3$ this was proved in
\cite{le1976invariance-milnor, timourian1977invariance-milnor}, and
for \eqref{eq:brieskorn-pham} with $n=2$ in
\cite{wirthmuller1978thesis}.

The vector space $W^-$ generated by the monomials of degree $<d$, i.e.\ with
$\delta_I>0$, defines \emph{the smoothing component} of $f_0$.
The polynomials $f_t$ in this space are exactly those that
can be completed to homogeneous degree-$d$ polynomials
\begin{equation}\label{eq:F_t}
  F_t =\sum_{k=0}^n x_k^{s_k} + \sum_{I\in\cI_n} t_I x^I \cdot
  x_{n+1}^{\delta_I}
\end{equation}
by adding a variable $x_{n+1}$ of degree~$1$, with the multi-indices
$I$ going over the set
\begin{equation}
  \label{eq:cI_n}
    \cI_n =
    \{ (i_0, \dotsc, i_n) \mid 0\le i_k\le s_k-2,\ \delta_I > 0 \}.
\end{equation}
We will call the equation \eqref{eq:F_t} and its affine version with
$x_{n+1}=1$ \emph{the normal form} or \emph{a short Weierstrass
  equation}.  Indeed, it is the classical short Weierstrass equation
if $n=1$.  The variety $X_t = (F_t=0)$ is a degree-$d$ hypersurface in the
weighted projective space
\begin{equation}
  \label{eq:P-nplus1}
  \amb := \bP^{n+1}(a_0, \dotsc, a_n, 1).
\end{equation}
We call $D_t = (x_{n+1}=0)\subset X_t$  \emph{the
  infinite divisor}.

An elementary argument shows that a general such hypersurface $X_t$ is an
$n$-dimensional Calabi-Yau variety with fixed-type canonical
singularities along the infinite divisor. These Calabi-Yau varieties
appeared in the work of Esser-Totaro-Wang \cite{esser2022calabi-yau}
in connection with their extremal properties. Singh
\cite{singh2025smooth-calabi} showed that $\amb$ admits a crepant
resolution of singularities, inducing crepant resolutions
$\wX_t\to X_t$ by smooth Calabi-Yau varieties~$\wX_t$ uniformly for a
general $t$.

From now on, we will restrict ourselves to the smoothing component
$W^-$ and consider only monomials with $\delta_I>0$.  Rescaling
$x_{k}\to \lambda^{a_k} x_{k}$ for $0\le k\le n$ and
$x_{n+1}\to x_{n+1}$ gives
\begin{equation}\label{eq:rescaling}
  F_t \to \lambda^d\Big(  \sum_{k=0}^n x_k^{s_k} + \sum_I
  (\lambda^{-\delta_I}t_I)x^I \cdot x^{\delta_I}_{n+1} \Big).
\end{equation}
Thus, replacing coefficients $(t_I)$ by $\big(\lambda^{\delta_I}t_I\big)$
gives an isomorphic hypersurface in~$\amb$. Up to this $\bG_m$-action,
$X_t$ with $t\ne0$ corresponds to a point of a quotient stack
\begin{equation}
  \label{eq:cym_n}
  \cym_n := \left[ \left(\bA^{N_1+1}\setminus 0\right) / \bG_m \right],
  \qquad (t_I) \to (\lambda^{\delta_I}t_I).
\end{equation}
This quotient stack is a weighted projective stack
$\cym_n=\cP^{N_1}(\delta_I,\, I\in\cI_n)$, and its coarse
moduli space is an ordinary weighted projective space
$\cymm_n=\bP^{N_1}( \delta_I,\, I\in\cI_n)$.  In particular, every
hypersurface $X_t$ is a \emph{small deformation} of $X_0$, so any
property of $X_0$ that is open in the Zariski topology automatically holds
for all $X_t$.
The dimension of $\cym_n$ is $|\cI_n| - 1$.
 For $n=1,\ 2,\ 3,\ 4,\ 5$ one has
 $\dim \cym_n= 1$, $10$, $251$, $151700$, $123769377141$.  Sylvester numbers
 grow doubly-exponentially: $s_n\approx c^{2^{n+1}}$, where
 $c=1.264\dots$. 
 Estimating as in~\eqref{eq:dim-cym}, we get
\begin{displaymath}
  \dim \cym_n \approx \frac{\mu}{(n-1)!} 
  \approx \frac{a}{(n-1)!} c^{2^{n+2}}, \quad
  \text{where } a=0.2789\dots 
\end{displaymath}

With $\amb$ and $\cym_n$ defined in \eqref{eq:P-nplus1} and \eqref{eq:cym_n}, our first result is:

\begin{maintheorem}\label{mainthm:family}
  For each $n\ge0$, working over $\bZ\big[1/\prod_{k=0}^n s_k \big]$,
  there is a flat family $\pi_n\colon (\cX_n,\cD_n)\to \cym_n$ of
  hypersurfaces $X\subset \amb$ together with $\bQ$-Cartier divisors
  $D\subset X$ in which every geometric fiber satisfies the following:
  \begin{enumerate}
  \item $X$ is an $n$-dimensional Gorenstein variety with a trivial
    dualizing sheaf $\omega_X$.
  \item $X\setminus D$ has only finitely many (none for a general $X$) isolated
    singularities.
  \item Along the infinite divisor $D$, $X$ has some
    $\mu_{\gcd(a_k,a_{k'})}$-quotient singularities, with $k\ne k'$, coming from the
    singularities of $\amb$, the same for every fiber.
  \item For $n\ne0$, $X$ is irreducible, and for $n\ne 1$, $X$ is normal.
  \end{enumerate}
  The relative dualizing sheaf of $\pi_n$ is
  $\omega_{\cX_n/\cym_n} = \pi_n^*\big(\cO(1)\big)$. Moreover, there
  is a natural injection $\iota_{n-1}\colon \cym_{n-1} \to \cym_n$.
  The restriction of the family $\pi_n$ to the boundary
  $\partial\cym_n := \iota_{n-1}(\cym_{n-1})$ and the family
  $\pi_{n-1}$ over $\cym_{n-1}$ are related in a natural way.
\end{maintheorem}

We give examples of the families $\pi_n$ and the embeddings
$\iota_{n-1}$ in Section~\ref{sec:En-1-En}.

For $n=1$, $\cym_1$ is the compactified moduli stack
$\overline{\cM}_{1,1}=\cP(4,6)$ of elliptic curves.

For $n=2$, generically $X$ is a K3 surface with $A_1A_2A_6$
singularities along the infinite divisor $D$. On the minimal
resolution $\wX$ the exceptional curves together with the strict
preimage of $D$ form a configuration of ten $(-2)$-curves spanning a
sublattice $\Lambda = U\oplus E_8\subset \Pic\wX$. The family $\pi_2$
was considered by Brieskorn
\cite{brieskorn1981unfolding-of-exceptional} who proved that
$\cymm_2 = \bP(4,10,\dotsc, 42)$ coincides with the Baily-Borel
compactification $\oF_\Lambda^{\rm BB}$ of the moduli space of
$U\oplus E_8$-polarized K3 surfaces.

The singularity $x_0^2+x_1^3+x_2^7$ is known as Arnold's exceptional
unimodal singularity $E_{12}$, and it also goes by the names
$S_{2,3,7}$ and $D_{2,3,7}$. By Arnold \cite{arnold1974normal-forms}
and Brieskorn \cite{brieskorn1979hierarchie}, its only deformations
are the cuspidal singularity $T_{2,3,7}$, the simple elliptic
singularity $\wE_8$, and $ADE$ singularities. In the language of
birational geometry, $S_{2,3,7}$ itself is not log canonical, while
$T_{2,3,7}$ and $\wE_8$ are strictly log canonical, and the $ADE$
singularities are canonical.  Our next result says that
this extends to dimension $n$. The singularity
\eqref{eq:brieskorn-pham} is not log canonical (cf.
Corollary~\ref{cor:brieskorn-pham-nonlc}), while its deformations
in the smoothing component satisfy the following:


\begin{maintheorem}\label{mainthm:singularities}
  Over $\bC$, for a geometric fiber $(X,D)$ of $\pi_n$ the
  following is true:
  \begin{enumerate}
  \item Near the infinite divisor $D\subset X$, $X$ has
    canonical singularities and the pair $(X,D)$ has log canonical
    singularities.
  \item For $[X] \in\partial\cym_n:=\iota_{n-1}(\cym_{n-1})$, $X$ has
    strictly log canonical singularities if $n\ne 1$, and a simple node
    (which is semi log canonical) if $n=1$.
  \item For $[X] \in\cym_n\setminus\partial\cym_n$, $X$ has canonical
    singularities. 
  \end{enumerate}
\end{maintheorem}

We give two moduli interpretations of the family $\pi_n$:

\begin{maintheorem}\label{mainthm:moduli-hypersurfaces}
  Over $\bZ\left[1/\prod_{k=0}^n s_k\right]$, $\cym_n$ is the moduli stack of
  Fermat-nondegenerate hypersurfaces of degree $d$ in $\amb$ modulo
  $\Aut(\amb, \bD)$, where $\bD=\{x_{n+1}=0\}$.
\end{maintheorem}

\begin{maintheorem}\label{mainthm:moduli-ksba}
  Over $\bC$, $\cym_n$ is a connected component of the moduli stack of
  KSBA-stable pairs $(X,D)$.
\end{maintheorem}

The moduli space $\cymm_n$ is a compactification of the open locus
$\cymm_n\setminus\partial\cymm_n$ parameterizing canonical
hypersurfaces $X$, and its boundary $\partial\cymm_n$ is tiny. In this way,
it is very similar to the Baily-Borel compactifications of the moduli
spaces $M_{1,1}$ and $F_\Lambda$ of elliptic curves and K3 surfaces
for $n=1$, $2$.
This observation prompts us to investigate Hodge theory of 
Calabi-Yau varieties appearing in the family $\pi_n$. Let 
\begin{equation}\label{eq:N_p}
  N_p = \# \big\{(i_0,\dotsc, i_n) \mid
   0\le i_k \le s_k-2, \ 
   p-1 < \sum_{k=0}^n \frac{i_k}{s_k} \le p \big\}.
 \end{equation}

 \begin{maintheorem}\label{mainthm:hodge-numbers}
   Let $[X]\in\cym_n$ be a Calabi-Yau variety such that $X\setminus D$ is
   smooth, and let $\wX\to X$ be a crepant resolution. Then
   $h^{p,p}(\wX) = h^{p,n-p}(\wX) = N_p$ except when $n$ is even and
   $p=n/2$, in which case $h^{p,p} = 2N_{p}$. In particular, for $n\ge3$
   $$h^1(T_\wX) = h^{n-1,1}(\wX) = N_1 = \dim\cym_n.$$
 \end{maintheorem}
 Note: by \cite{esser2022calabi-yau} one has $h^{p,q}(\wX)=0$ unless
 $p=q$ or $p+q=n$.
 
 \begin{maintheorem}\label{mainthm:kodaira-spencer}
   Assume $n\ge3$ and
   let $\wX$ be as in Theorem~\ref{mainthm:hodge-numbers}. Then the
   Kodaira-Spencer map $T_{\cym_n, [X]}\to H^1(T_\wX)$ is an isomorphism.
 \end{maintheorem}
 It is well known that for Calabi-Yau varieties the injectivity of the
 Kodaira-Spencer map implies a local Torelli theorem for the family $\pi_n$.

 \medskip

 The plan of the paper is as follows. In
 Section~\ref{sec:preliminaries} we recall basic facts about Sylvester
 numbers and about canonical and log canonical singularities.

 In Section~\ref{sec:family} we define the family $\pi_n$ and
 prove Theorem~\ref{mainthm:family}  and Theorem~\ref{mainthm:singularities}(1).

 In Section~\ref{sec:sings-nondegenerate} we prove
 Theorem~\ref{mainthm:singularities} about (log)canonical
 singularities in the special case when these singularities are
 Newton-nondegenerate, in which case this becomes a purely
 combinatorial property of their Newton polyhedra. This is a
 prerequisite for the general case and it illustrates what is so
 special about the Sylvester sequence.

 In Section~\ref{sec:fibered-varieties} we study varieties generically
 fibered in Calabi-Yau varieties from~$\cym_n$. For $n=1$ this is the
 classical theory of elliptic surfaces with a section, and the
 $n$-dimensional case turns out to be quite similar. We extend to
 dimension $n$ some classical results, such as: the minimal form, a
 criterion for canonical singularities, a formula for the log
 canonical threshold of a fiber, the canonical class formula, Tate's
 algorithm, and Kodaira's classification of degenerate elliptic fibers.

 In Section~\ref{sec:sing-general} we show that for any pair
 $[(X,D)]\in \cym_n$, its affine part $X\setminus D$ is fibered over
 $\bA^1$, with a generic fiber equal to the affine part of a Calabi-Yau
 hypersurface in $\cym_{n-1}$. Using induction and results of
 Section~\ref{sec:fibered-varieties}, we prove
 Theorem~\ref{mainthm:singularities} in general.

 In Section~\ref{sec:moduli} we prove
 Theorems~\ref{mainthm:moduli-hypersurfaces} and~\ref{mainthm:moduli-ksba}
 about the moduli interpretations of~$\cym_n$.

 Finally, in Section~\ref{sec:hodge-theory} we prove
 Theorems~\ref{mainthm:hodge-numbers},~\ref{mainthm:kodaira-spencer}
 and the local Torelli theorem.

 In Section~\ref{sec:wps} we work over $\bZ$, in
 Sections~\ref{sec:def-of-family},~\ref{sec:En-1-En}, and~\ref{sec:moduli-hypersurfaces} over $\bZ[1/\prod_{k=0}^n s_k]$, and
 in the rest of the paper over $\bC$.

\subsection*{Table of notations}

\renewcommand{\arraystretch}{1.2}
\begin{longtable}{ll}
  $s_k$ & Sylvester numbers $s_0=2$, $s_{n+1} = 1 + \prod_{k=0}^{n}
  s_k$ \\
  $f_0$ & the polynomial $\sum_{k=0}^{n} x_k^{s_k}$\\
  $\mu$ & Milnor number of $f_0$, equal to $\prod_{k=0}^n (s_k-1)$\\
  $d,\ d_n$  & weighted degree of $f_0$, equal to $\prod_{k=0}^n s_k$\\     
  $a_k,\ a_{n,k}$ & degree of $x_k$, equal to $d/ s_k$\\
  $I$ & multi-index $(i_0, \dotsc, i_n)$\\
  $\delta_I$ & degree difference of $x^I$, equal to $d - \sum_{k=0}^n a_{k} i_k =d \Big( 1 - \sum_{k=0}^n \frac{i_k}{
    s_k} \Big)$ \\
  $\cI_n$ & the index set
  $\{ (i_0, \dotsc, i_n) \mid 0\le i_k\le s_k-2,\ \delta_I > 0 \}$\\
  $\cI_n^+$ & the extended index set
  $\{ (i_0, \dotsc, i_n) \mid 0\le i_k\le s_k-1,\ \delta_I > 0 \}$\\
  $f_t$ & affine equation $\sum_{k=0}^{n} x_k^{s_k} + \sum t_Ix^I$\\
  $F_t$ & homogenization of $f_t$, equal to $\sum_{k=0}^{n} x_k^{s_k} +
  \sum t_Ix^I \cdot x_{n+1}^{\delta_I}$\\
  $\amb$ & weighted projective space $\bP(a_0, \dotsc, a_n, 1)$\\
  $\bD,\ \bD_t$ & ``infinite divisor'' $(x_{n+1}=0)\subset \amb$, isomorphic
  to $\bP^n$\\
  $X,\ X_t$ & degree-$d$ hypersurface $(F_t=0)\subset\amb$\\
  $D,\ D_t$ & ``infinite divisor'' $(x_{n+1}=0)\subset X$, isomorphic
  to $\bP^{n-1}$\\
  $\cym_n$ & weighted projective stack $\cP(\delta_I,\ I\in\cI_n)$\\
  $\cymm_n$ & coarse moduli space of $\cym_n$, a weighted projective space $\bP(\delta_I,\ I\in\cI_n)$\\
  $\pi_n$ & the universal family $(\cX_n,\cD_n) \to \cym_n$\\
  $\nuu_p$ & minimal reduced valuation for a family $X\to (C,p)$ over
  a curve
\end{longtable}

\ifthenelse{\boolean{springer}}{%
  \begin{acknowledgements}
    The author would also like to thank Jas Singh for useful
    references.
  \end{acknowledgements}

\subsection*{Statements and declarations}
  The author was partially supported by the NSF under DMS-2501855.
  The author declares no competing interests.
}
{%
  \begin{acknowledgements}
    The author was partially supported by the NSF under DMS-2501855.
    He would also like to thank Jas Singh for useful references.
  \end{acknowledgements}
}

\presectionskip
\section{Preliminaries}
\label{sec:preliminaries}

\subsection{Sylvester numbers}
\label{sec:sylvester-numbers}

The Sylvester sequence 
$\{s_0,s_1,s_2,s_3,\dotsc \} = \{2,3,7,43,\dotsc\}$ is defined by the
recursive relation
\begin{math}
  s_{n+1} = 1 + \prod_{k=0}^{n} s_k.
\end{math}

\begin{definition}\label{def:d-a-nk}
  Let
  \begin{equation}
    \label{eq:d-a-nj}
    d_n= \prod_{k=0}^ns_k, \qquad
    a_{n,k}=\frac{d_n}{s_k}=\prod_{\substack{0\le i\le n\\  i\ne k}} s_i.
    \ \text{ for } 0\le k\le n. 
  \end{equation}
  When $n$ is clear from context, we abbreviate these numbers as $d$ and
  $a_k$.
\end{definition}


\begin{lemma}\label{lem:sylvester-properties}
  $s_k$ are pairwise coprime, $\gcd(a_0,\dotsc,a_n) = 1$ if
  $n>0$. They satisfy
  \begin{eqnarray}
    \label{eq:egyptian}
    &\text{(Egyptian fraction identity)} &
          \sum_{k=0}^{n}\frac1{s_k} + \frac1{d} = 1
          \iff \sum_{k=0}^n a_{k} + 1 = d \\
    \label{eq:congruence}
    &\text{(Congruence relation)} & a_{k} \equiv -1\mymod{s_k}\\
    \label{eq:linear-relation}
    &\text{(The $\gcd$ relation)} & a_0 - \sum_{k=1}^n a_{k} =1
  \end{eqnarray}
\end{lemma}
\begin{proof}
  \eqref{eq:egyptian} is easily proved by induction on $n$.  From the
  defining equation for $s_k$, we have
  $\prod_{i<k} s_i \equiv -1\mymod{s_k}$ and $s_i \equiv 1\mymod{s_k}$
  for $i>k$, implying
  \eqref{eq:congruence}. \eqref{eq:linear-relation} follows from
  \eqref{eq:egyptian} and $d=2a_0$.
\end{proof}

\subsection{Canonical and log canonical singularities}
\label{sec:canonical-sings}

We refer the reader to standard sources such as
\cite{kollar1998birational-geometry, kollar2013singularities-mmp} for a discussion of
singularities appearing in the Minimal Model Program. For
the reader's convenience, we recall some of them below.

\begin{definition}\label{def:X-sings}
  Let $X$ be a normal variety and assume that the canonical class
  $K_X$ is $\bQ$-Cartier. For any (partial) resolution of singularities
  $f\colon Y\to X$ with normal variety $Y$ there is a natural formula
  \begin{displaymath}
    K_Y = f^*(K_X) + \sum a(X,E_i) E_i, \quad E_i: \text {$f$-exceptional divisors}.
  \end{displaymath}
  The numbers $a(X,E_i)\in\bQ$ are called discrepancies. One says that
  $X$ is terminal, resp.\ canonical, resp.\ klt, resp.\ log canonical
  (lc) if for any $f$ one has $a(X,E_i)>0$, resp.\ $a(X,E_i)\ge0$,
  resp.\ $a(X,E_i)>-1$, resp.\  $a(X,E_i)\ge-1$. Another way to define
  canonical singularities is to say that $f^*(K_X) \le K_Y$ for any
  $f$.
\end{definition}

For most of the paper, $K_X$ will be Cartier. Then $a(X,E_i)\in\bZ$ and klt
is equivalent to canonical. Canonical surface singularities are the
$ADE$ singularities (also called rational double points, or Du Val
singularities), and the strictly log canonical singularities are
simple elliptic and cusp singularities.
A particularly nice class of canonical singularities is those that 
admit crepant resolutions. We recall:

\begin{definition}\label{def:crepant}
  One says that a proper birational morphism $\phi\colon Y\to X$ is
  crepant if $K_Y = \phi^*(K_X)$.
\end{definition}

Definition~\ref{def:X-sings} extends to the pairs:

\begin{definition}
  Let $X$ be a normal variety again, and let $B=\sum b_jB_j$ be a
  $\bQ$-Weil divisor with $b_j\in\bQ$, $b_j\le 1$.
  The divisor $B$ is frequently called a sub-boundary, and it is
  called a boundary if all $b_j\ge0$.
  Assume that $K_X+B$ is $\bQ$-Cartier. For any (partial) resolution of
  singularities $f\colon Y\to X$ there is a natural formula
  \begin{displaymath}
    K_Y + f_*\inv(B) = f^*(K_X+B) + \sum a(X,B,E_i) E_i, \quad E_i: \text {$f$-exceptional divisors}.
  \end{displaymath}
  One says that the pair $(X,B)$ is klt, resp.\ lc if for any $f$ one
  has $b_j<1$ and $a(X,B,E_i)>-1$, resp.\ $b_j\le 1$ and
  $a(X,B,E_i)\ge -1$.  Another way to define log canonical
  singularities is to say that for any $f$ one has
  $$f^*(K_X+B) \le K_Y + f_*\inv(B) + \sum E_i.$$
\end{definition}

There is a generalization of log canonical singularities to the nonnormal case:
\begin{definition}\label{def:slc}
  The pair $(X,B)$ is semi log canonical (slc) if $X$ satisfies
  Serre's condition $S_2$ and has only ordinary double crossing
  singularities in codimension~$1$, $K_X+B$ is $\bQ$-Cartier, and for
  the normalization $\nu\colon X^\nu\to X$ with the double locus
  $D^\nu$, the pair $(X^\nu, D^\nu+\nu_*\inv(B))$ is lc.
\end{definition}

We will also need the following notion:
\begin{definition}\label{def:lct}
  Let $X$ be a $\bQ$-Gorenstein variety and $D$ an effective $\bQ$-Cartier
  divisor on $X$. The log canonical threshold is
  \begin{displaymath}
    \lct(X,D) = \sup \{t\in\bR \mid (X,tD) \text{ is lc} \},
  \end{displaymath}
  with the convention that $\sup\emptyset=-\infty$. Note that it can
  be negative.
\end{definition}
Here are some basic properties that we will use, all of them well known.

\begin{theorem}[Inversion of Adjunction,
  \cite{kawakita2007inversion-of-adjunction}]
  \label{thm:inversion-of-adjunction}
  Let $X$ be a normal variety, $D$ a reduced Weil divisor,
  and $B=\sum b_iB_i$ a $\bQ$-divisor with $B_i$ and $D$ without
  irreducible components in common. Assume that $K_X+D+B$ is
  $\bQ$-Cartier and that $D$ is Cartier in codimension~$2$ on
  $X$. Then
  $(X,D+B)$ is lc in a neighborhood of $D$ iff $(D, B|_D)$
  is slc. 
\end{theorem}


\begin{theorem}
  \label{thm:lc-under-maps}
  Let $f\colon Y\to X$ be a generically finite surjective proper
  morphism of normal varieties. Let $B$ be a $\bQ$-divisor such that
  $K_X+B$ is $\bQ$-Cartier, and define the divisor $B^Y$ on $Y$ by the
  formula $K_Y+B^Y = f^*(K_X+B)$. Then $(X,B)$ is log canonical iff
  $(Y,B^Y)$ is log canonical.
\end{theorem}
\begin{proof}
  If $f$ is birational, this follows directly from the definition. For
  a finite morphism, this is
  \cite[Cor.~2.2]{shokurov1992three-dimensional} and
  \cite[Prop.~5.20(4)]{kollar1998birational-geometry}. The general
  case is a combination of these two cases and the Stein factorization.
\end{proof}

\begin{corollary}\label{cor:quotient-of-klt}
  A quotient $Y=X/G$ of a log canonical (resp.\ klt) variety by a
  finite group is log canonical (resp.\ klt). Similarly, a quotient
  $(Y,D_Y) = (X,D_X)/G$ of a log canonical pair with a reduced divisor
  $D_Y$ is log canonical.
\end{corollary}
\begin{proof}
  Let $f\colon X\to Y$ be the quotient map. One has
  $K_X = f^*(K_Y+B)$, where $B=\sum \frac{n_i-1}{n_i} B_i$ is the branch
  divisor. By Theorem~\ref{thm:lc-under-maps}, the pair $(Y,B)$ is log
  canonical (resp.\ klt), hence so is $(Y,0)$ by monotonicity since
  $B\ge 0$. For the pairs, the proof is similar, using the formula
  $K_X + D_X = f^*(K_Y+D_Y+B)$.
\end{proof}

\presectionskip
\section{The family $\pi_n$ and its basic properties} 
\label{sec:family}

\subsection{Weighted projective space and weighted projective stack}
\label{sec:wps}

We will use several versions of a weighted projective space in this paper:
\begin{enumerate}
\item A weighted projective space $\amb=\bP(a_0,\dotsc, a_n,1)$.
\item A weighted projective stack $\cym_n=\cP\big(\delta_I)$, for
  $I\in\cI_n$ defined in \eqref{eq:cI_n}.
\item A weighted projective space $\cymm_n=\bP\big(\delta_I,\ I\in\cI_n)$.
\item A weighted projective $\amb$-bundle
  $\amb(V)=\Proj\,\Sym^*(V)\to S$ over a scheme.
\item A weighted projective $\amb$-bundle
  $\amb(\cV)=\Proj\,\Sym^*(\cV)\to \cym_n$ over a stack. 
\end{enumerate}

The most basic reference for $\Proj\,\Sym^*(V)$ is \cite{EGA2}. Some
convenient references for weighted projective spaces include
\cite{dolgachev1982weighted-projective,
  fletcher2000working-with-weighted} over a field and
\cite{delorme1975espaces-projectifs,amrani1989classes-ideaux} over a general ring. Likewise,
weighted projective stacks have been well discussed in the
literature. One convenient reference is
\cite{abramovich2011stable-varieties}.

\subsubsection{$\amb=\bP(a_0,\dotsc, a_n,1)$}

For any commutative ring $A$ with identity and positive integers
$b_0, \dotsc, b_N$, a weighted projective space $\bP_A(b_k)$
over $A$ is defined as $\Proj A[x_0, \dotsc, x_N]$, with the grading
given by $\deg x_k=b_k$ for $0\le k\le N$.  Like any $\Proj$, it comes
with the sheaves $\cO(m)$ for all $m\in\bZ$.  It is \emph{well-formed}
if after removing any one weight, the $\gcd$ of the remaining weights
is $1$.

\begin{lemma}[\cite{amrani1989classes-ideaux}, Sections 2 and 4]
  \label{lem:wps-U}
  Assuming well-formedness, the smooth locus
  $i\colon U\subset \bP_A(b_k)$ of
  $\bP_A(b_k)\to\Spec A$ is the maximal open subset such that
  $\cO_{U}(1)$ is invertible and $\cO_{U}(m) = \cO_{U}(1)^{\otimes m}$
  for all $m$. One has $\codim \bP_A(b_k)\setminus U\ge 2$ and
  $i_*\cO_U(m) = \cO(m)$.  The sheaf $\cO(m)$ is invertible iff all
  $\deg x_k$ divide $m$.
\end{lemma}

\begin{lemma}[\cite{amrani1989classes-ideaux}, 6.2]
  \label{lem:wps-cohs}
  Assuming well-formedness, the canonical homomorphism
  \begin{equation}
    \label{eq:wps-sectionring}
    A[x_0, \dotsc, x_N] \to 
    \bigoplus_{m\in\bZ} H^0(\bP_A(b_k), \cO(m))
  \end{equation}
  is an isomorphism of graded $A$-algebras. One has
  \begin{eqnarray*}
    &H^i\big(\bP_A(b_k), \cO(m)\big) = 0\quad\text{for}\quad i\ne 0,\ N
    \\
    &H^{N}\big(\bP_A(b_k), \cO(m)\big) =
      H^0\big(\bP_A(b_k), \cO(-m -\sum_{k=0}^N b_k)\big)^\vee.
  \end{eqnarray*}
\end{lemma}

We define $\amba:=\bP_A(a_0,\dotsc, a_n,1)$. For $n>0$, it is well-formed since
$\gcd(a_k)=1$. 

\begin{lemma}\label{lem:P-Gorenstein}
  $\amba$ is Gorenstein over $\Spec A$. The dualizing sheaf
  $\omega_{\amba/A}\simeq\cO(-d)$ is invertible and relatively anti-ample.
\end{lemma}
\begin{proof}
  By the Egyptian fraction identity~\eqref{eq:egyptian}, one has
  $$\omega_{\amba/A} = \cO\big(-\sum_{k=0}^n a_k-1\big) = \cO(-d)$$
  This is an invertible sheaf since all $a_k$ divide $d$, and it is
  clearly anti-ample.
\end{proof}

\begin{lemma}\label{lem:bD-in-bP}
  The infinite divisor $\bD_A = (x_{n+1}=0) \subset \amb_A$ is
  isomorphic to an ordinary projective space $\bP^n_A$.
\end{lemma}
\begin{proof}
  $\bD_A = \bP_A(a_0, \dotsc, a_n)$ is not well-formed since
  $\gcd(a_i,\ i\ne k)=s_k$. Then
  $\bP_A(a_0,a_1, \dotsc, a_n) = \bP_A(a_0, a_1/s_0, \dotsc, a_n/s_0)
  = \dotsb =\bP_A(1,\dotsc,1)$. Explicitly, the $d$-th Veronese
  subalgebra $A[x_0,\dotsc, x_n]^{(d)}$ is isomorphic to
  $A[u_0,\dotsc, u_n]$, with the generators $u_k= x_k^{s_k}$. So 
  $\bP(a_0, \dotsc, a_n) \simeq \Proj A[x_0,\dotsc, x_n]^{(d)} = \bP^n_A$.
\end{proof}

\subsubsection{$\cP(\delta_I)$}

The weighted projective stack ${\cym_n}$ is defined as the tame
quotient stack
\begin{equation}
  \label{eq:wpstack}
  {\cym_n}:= \cP(\delta_I) = \left[ \left(\bA^{N_1+1}\setminus 0\right) / \bG_m
  \right],
  \qquad (t_I) \to (\lambda^{\delta_I}t_I),
\end{equation}
with $I\in\cI_n$ as in~\eqref{eq:cI_n}.  For any $m\in\bZ$ there is an
invertible sheaf $\cO(m) = \cO(1)^{\otimes m}$ on ${\cym_n}$. This
definition works over any base scheme, for example over
$\Spec\bZ$. 

\begin{lemma}\label{lem:generic-inertia}
  The generic inertia group of the stack $\cym_n$ is $\mu_2$. 
\end{lemma}
\begin{proof}
  The generic inertia group of a weighted projective stack
  $\cP(\delta_I)$ is $\mu_g$, where $g = \gcd(\delta_I)$. Since for
  the multi-indices $I$ of \eqref{eq:F_t} one has $0\le i_k\le s_k-2$,
  $s_0=2$ and $s_k>2$ for $k>0$, only $a_k$ for $k>0$ appear in the
  definition \eqref{eq:degree-diff} of $\delta_I$. Thus,
  \begin{displaymath}
    \gcd(\delta_I) = \gcd(a_k,\ k>0) = s_0 \gcd(s_k,\ k>0) = s_0 = 2.
  \end{displaymath}
  The statement follows.
\end{proof}

\subsubsection{$\bP(\delta_I)$}

The coarse moduli space $\cymm_n$ of $\cym_n$ is an ordinary
projective space $\bP(\delta_I,\ I\in\cI_n)$.  There is a natural map
$\cP(\delta_I)\to\bP(\delta_I)$, and for all $m\in\bZ$ one has
$p_*\cO_{\cP(\delta_I)}(m) = \cO_{\bP(\delta_I)}(m)$.

\subsubsection{$\amb(V)$}
\label{sec:bPV}

Let $L$ be a line bundle on a scheme $S$. By analogy with elliptic
surfaces, we will call it a \emph{fundamental line bundle}. We define
a graded vector bundle $V = \oplus_{k=0}^{n+1}V_k$, where
$V_k=L^{-a_k}$ for $0\le~k\le~n$ is a line bundle of weight $a_k$ and
$V_{n+1}=\cO_S$ has weight~$1$.  The $\amb$-bundle $\amb(V)$ over a
scheme $S$ is defined by choosing an open cover $\{U_i = \Spec A_i\}$
on which $L$ trivializes and gluing $\amb_{A_i}$'s.

Alternatively, $\amb( V) \simeq \amb( V')$ for
$ V' = \oplus_{k=0}^{n+1} V'_k$, where $V'_k = \cO_{S}$ for
$0\le k\le n$ and $V'_{n+1} = L\inv$, i.e.\ 
$ V_k = V'_k \otimes L^{\deg V_k}$. One has
$\cO_{\amb( V)}(1) = \cO_{\amb( V')}(1) \otimes \pi^*( L)$.

\subsubsection{$\amb(\cV)$}

Finally, there is a straightforward extension of the
$\Proj\,\Sym^*(V)$ construction from the base scheme to a base stack,
such as $\cym_n$, explained for example in
\cite[\S14.3]{laumon2000champs-algebriques}.  We take
$\cV = \oplus_{k=0}^{n+1}\cV_k$ to be a graded vector bundle on
$\cym_n$, where $\cV_k=\cO_{\cym_n}(-a_k)$ for $0\le~k\le~n$ with
$\deg\cV_k=a_k$ and $\cV_{n+1}=\cO_{\cym_n}$ with $\deg\cV_{n+1}=~1$.
This $\amb$-bundle is
locally trivial in fppf topology.  Over $\bQ$ it is also locally
trivial in \'etale topology since ${\cym_n}$ can be also written as a
quotient stack $[\bP^{N_1} / \prod_I \mu_{\delta_I}]$ and in
characteristic $0$ the groups $\mu_{\delta_I}$ are \'etale.

\subsection{The family of hypersurfaces}
\label{sec:def-of-family}

\begin{definition}
  \label{def:family-hypersurfaces}
  For any line bundle $L$ on a scheme $S$, let $\pi\colon\amb(V)\to S$
  be the weighted $\bP$-fibration over $S$ as in the previous section.
  For any global sections $t_I\in H^0\big(S, L^{\delta_I}\big)$, the
  polynomial $F_t$ of Equation~\eqref{eq:F_t} defines a hypersurface
  $\cX\subset \bP(V)$, for which $x_k$ and $F_t$ are sections on
  $\amb( V)$ of
\begin{displaymath}
    x_k:\,  \cO(a_k) \otimes \pi^*\big( L^{a_k}\big)
    \text{ for } 0\le k\le n, \quad
    x_{n+1}:\,  \cO(1),
    \quad
    F_t:\, \cO(d) \otimes \pi^*\big( L^d\big)
\end{displaymath}
\end{definition}

Note that $\pi_*\cO(d) = \Sym^d( V)$, which contains a copy of
$\otimes_{k=0}^n  V_k^{\otimes i_k}$ for every monomial $t_Ix^I$ in
$F_t$. This implies that $H^0\big(\cO(d) \otimes
\pi^*\big( L^d\big)\big) \ni F_t$ contains a copy of 
$H^0(S,  L^{d-\sum a_ki_k }) = H^0(S,  L^{\delta_I}) \ni t_I$.

\begin{lemma}\label{lem:basic-family-canclass}
  $\cX$ is flat and Gorenstein over $S$, and
  $\omega_{\cX/S} \simeq \pi^*( L)$.
\end{lemma}
\begin{proof}
  Indeed, by Lemma~\ref{lem:P-Gorenstein},
  $\amb( V)$ is Gorenstein over $S$ and $\cX$ is a relative Cartier
  divisor. So $\cX$ is Gorenstein over $S$. One has
  \begin{displaymath}
    \omega_{\amb( V)/S} =
    \cO\Big(-\sum_{k=0}^n a_k - 1\Big) \otimes \pi^*(\det  V) = 
    \cO(-d) \otimes \pi^*\big( L^{-d+1}\big) 
  \end{displaymath}
  By adjunction,
  \begin{displaymath}
    \omega_{\cX/{S}} = \omega_{\amb( V)/S}
    \otimes
    \cO(d) \otimes \pi^*\big( L^d\big)
    |_{\cX}=
    \pi^*( L)
  \end{displaymath}
\end{proof}

\begin{lemma}
  Assuming $n>0$, one has $L = \pi_*(\omega_{\cX/S})$
\end{lemma}
\begin{proof}
  By Lemma~\ref{lem:cohs-cX}, 
  $\pi_*(\cO_{\cX/S}) = \cO_S$. The projection formula gives
  $\pi_*(\omega_{\cX/S}) = \pi_*(\pi^* L) =  L \otimes
  \pi_*(\cO_{\cX/S})~=~L$.
\end{proof}

As it is customary in similar situations, one may call $L$ \emph{the
  Hodge bundle}.

\begin{lemma}\label{lem:cO_X(m)-are-flat}
  For any $m\in\bZ$ the sheaf $\cO_\cX(m)$ is flat over $S$.
\end{lemma}
\begin{proof}
  We can work locally over $\Spec A\subset S$ and assume
  $L=\cO_{\Spec A}$. The ring $A[x_0,x_1,\dotsc, x_{n+1}]/F_t$ is a
  free rank-$2$ $A[x_1,\dotsc, x_{n+1}]$-module with a basis
  $\{1,x_0\}$, so it is a priori a free $A$-module. The sheaf
  $\cO_\cX(m)$ corresponds to the module obtained from this ring by
  localizing in homogeneous positive-degree polynomials and taking the
  degree-$m$ direct summand. So it is flat over $A$.
\end{proof}

\begin{lemma}\label{lem:cohs-cX}
  Assuming $n>0$, one has $R^i\pi_*\cO_\cX(m)=0$ for $i>0$ and
  $m\ge 0$, and there is an equality of graded $\cO_S$-algebras
  $\oplus_{m\ge0} \pi_* \cO_\cX(m) = \Sym^*( V) / F_t\cdot
  \Sym^*( V)$.
\end{lemma}
\begin{proof}
  Immediate from the exact sequence
  \begin{displaymath}
    0\to \cO_{\amb( V)}(m-d)\to \cO_{\amb( V)}(m) \to \cO_\cX(m) \to 0
  \end{displaymath}
  and the vanishing in Lemma~\ref{lem:wps-cohs}.
\end{proof}

\begin{definition}
  In particular, one has
  $H^0\big(\cX, \cO(1)\big) = H^0(S,  V_{n+1}) = H^0(S, \cO_S)$.  We
  define a relative divisor $\cD$ on $\cX$ as the zero divisor of the
  section $x_{n+1}$, corresponding to $1_S$. Explicitly, this is the
  closure of the Cartier divisor on $\cX$ intersected with the smooth
  locus of $\bP/S$.  The multiple $d\cD$ is the zero divisor of a
  section of the line bundle $\cO_\cX(d)$, and it is a relative
  Cartier divisor. Thus, $\cD$ is a relative $\bQ$-Cartier divisor.
\end{definition}

\begin{lemma}
  Under the identification of the infinite divisor $\bD_A\subset\amb_A$
  with $\bP_A^n$ of Lemma~\ref{lem:bD-in-bP}, the infinite divisor
  $\cD_A\subset\cX_A$ is isomorphic to a hyperplane $\bP_A^{n-1}$.
\end{lemma}
\begin{proof}
  Indeed, under the identification of Lemma~\ref{lem:bD-in-bP},
  $\sum_{k=0}^n x_k^{s_k} = \sum_{k=0}^n u_k$.
\end{proof}
\medskip

\begin{definition}\label{def:universal-hypersurface}
  We define the universal hypersurface $\pi_n\colon\cX_n\to\cym_n$ in
  the weighted projective $\amb$-bundle $\pi\colon\amb(\cV)\to\cym_n$ over the
  weighted projective stack ${\cym_n}=\cP(\delta_I)$ by the polynomial
  $F_t$ of Equation \eqref{eq:F_t}.  Here, the $t$-variables are
  global sections $t_I\in H^0({\cym_n}, \cO(\delta_I))$. The $x$-variables
  and $F_t$ are global sections on $\amb(\cV)$ of
\begin{displaymath}
  x_k:\, \cO(a_k) \otimes \pi^*\cO_{\cym_n}(a_k)
  \text{ for } 0\le k\le n, \quad
  x_{n+1}:\, \cO(1),
    \qquad
    F_t:\, \cO(d) \otimes \pi^*\cO_{\cym_n}(d).
\end{displaymath}
   We define a relative divisor $\cD_n$ on $\cX_n$ as the
  zero divisor of the section
  $x_{n+1} \in H^0\big(\amb(\cV), \cO(1)\big)$. It is a relative
  $\bQ$-Cartier divisor.
\end{definition}

The following is clear:

\begin{lemma}\label{lem:map-to-stack}
  Let $\pi\colon \cX\to S$ be a family of hypersurfaces from
  Definition~\ref{def:family-hypersurfaces}. If the sections $t_I$ do
  not simultaneously vanish at any point $s\in S$ then $\pi$ is a
  pullback of $\pi_n\colon\cX_n\to \cym_n$ for a unique morphism
  $g\colon S\to\cym_n$, and one has $L=g^*(\cO_{\cym_n}(1))$.
\end{lemma}

\subsection{The embedding $\cym_{n-1}\to \cym_n$}
\label{sec:En-1-En} 

The stack $\cym_{n-1}$ is embedded into $\cym_n$ by
\begin{equation}\label{eq:cym-n-1-to-cym-n}
  g_{t'} = \sum_{k=0}^{n-1}x_k^{s_k}
  + \sum_{I'} t'_{I'}x^{I'} \cdot x_n^{\delta_{I'}}
  \ \longmapsto\ 
  f_{t} = \sum_{k=0}^{n-1} x_k^{s_k}
  + \sum_{I'} t'_{I'}x^{I'} \ell^{\delta_{I'}} \cdot x_{n+1}^{\delta_{I'}}
  + \ell^{s_n},
\end{equation}
where $I'=(i_0,\dotsc, i_{n-1})$,
$\delta_{I'} = d_{n-1}\big(1 - \sum_{k=0}^{n-1} \frac{i_k}{s_k}\big)$,
and $\ell=x_n - \frac1{s_n} t'_{0,\dotsc,0} \cdot
x_{n+1}^{d_{n-1}}$. This map can be split into several steps:
\begin{enumerate}
\item Take a cone over $g_{t'}$, i.e.\ consider $g_{t'}$,
  which is a quasi-homogeneous
  equation of degree $d_{n-1}=s_n-1$
  in weighted \emph{projective} coordinates $x_0,\dotsc, x_n$,
  as an equation in \emph{affine} coordinates $x_0,\dotsc, x_n$.
\item Make a small deformation by adding the term
  $x_n^{s_n}$ of degree one higher.
\item On the affine chart $x_{n+1}=1$, make a linear change of coordinates
  $x_n\to \ell = x_n-\frac1{s_n}t'_{0,\dotsc,0}$ to eliminate the
  $x_n^{s_n-1}$ term (``complete the $s_n$-th power'').
\item And finally make it into a homogeneous equation of degree $d=d_n$ by
  adding the variable $x_{n+1}$ of degree $1$.
\end{enumerate}

\begin{remark}\label{rem:constant-term}
  In Equation~\eqref{eq:cym-n-1-to-cym-n} the coefficient of
  $x_{n+1}^{d_n}$ in $f_t$, i.e.\ the constant term of $f_t$ on the
  affine chart $\{x_{n+1}=1\}$, comes from
\begin{displaymath}\textstyle
  \ell^{s_n-1}(\ell + t'_{0,\dotsc,0} x_{n+1}^{d_{n-1}}) =
  (x_n - \frac1{s_n} t'_{0,\dotsc,0} x_{n+1}^{d_{n-1}})^{s_n-1}
  (x_n + \frac{s_n-1}{s_n} t'_{0,\dotsc,0} x_{n+1}^{d_{n-1}}) 
\end{displaymath}
and equals $(s_n-1) \big( t'_{0,\dotsc,0} / s_n \big)^{s_n} =
d_{n-1} \big( t'_{0,\dotsc,0} / s_n \big)^{s_n}$. 
\end{remark}

\begin{definition}
  We will call the image of $\cym_{n-1} \to \cym_n$ \emph{the
    boundary} $\pd\cym_n$ of $\cym_n$.
\end{definition}

\begin{definition}\label{def:level}
  For a variety $[X]\in\cym_n$, the \emph{level} of $X$ is the maximal
  integer $c$ such that $[X]\in\im\cym_{n-c}$.
\end{definition}

In the examples below, to save space, we write $t_{i_1, \dotsc, i_n}$
instead of $t_{i_0,i_1, \dotsc, i_n}$ for the coordinates, since the condition
$0\le i_0 \le s_0-2 = 0$ forces $i_0=0$ everywhere.

\begin{example}\label{ex:dim0}
  $\cym_0=\bP(2)$ with the coordinate $t$ and
  \begin{displaymath}
    \cX^{(0)} = \{x_0^2 + t \cdot x_1^2 = 0 \} \subset \cym_0 \times
    \bP(1,1) = \cym_0\times\bP^1. 
  \end{displaymath}
  As a set, $\cym_0$ is a single point and
  as a stack, $\cym_0=B\mu_2$, reflecting the fact that the automorphism
  group of $[X]\in\cym_0$ (a union of two points) is~$\mu_2$. 
\end{example}

\begin{example}\label{ex:dim1}
  $\cym_1=\bP(4,6)$ with the coordinates $t_{1}$, $t_{0}$ and
  \begin{displaymath}
    \cX^{(1)} = \{ x_0^2 + x_1^3 + t_{1} x_1 \cdot x_2^4 + t_{0}
    \cdot x_2^6  = 0\}
    \to \cym_1
  \end{displaymath}
  is a family of hypersurfaces in $\bP(3,2,1)$.
  $\cym_0$ is embedded into $\cym_1$ by
  \begin{displaymath}
    x_0^2 + t\cdot x_1^2 \longmapsto x_0^2 + t \ell^2 \cdot x_2^2 + \ell^3,
    \quad\text{where } \ell = x_1 - \tfrac13 t \cdot x_2^2.
  \end{displaymath}
  Explicitly, this means that
  $t_{1} = -\tfrac13 t^2$ and $t_{0} = \tfrac2{3^3}t^3$.
  
  Of course, $\cym_1$ is $\oM_{1,1}$, the moduli stack of stable
  $1$-pointed genus-$1$ curves (also called the $j$-line),
  $j(t_{1},t_{0}) = 12^3 \frac{ 4t_{1}^3 } { 4t_{1}^3 + 27t_{0}^2 }$,
  and $\im\cym_0\subset \cym_1$ is the point where $j=\infty$.  Note
  that the singularity of the singular curve in $\cym_1$ (a node) is a
  cone over $X_0\in\cym_0$.  The divisor $D=(x_2=0)\subset X$ is the
  point at infinity, and $\Aut(X,D)=\mu_2$ for a general $X$,
  explaining why the weights $4,6$ in $\cym_1$ are even. At the points
  $(t_0,t_1)=(0,1)$ and $(1,0)$ the automorphism groups of elliptic
  curves are $\mu_4$ and $\mu_6$ respectively; their $j$-invariants
  are $1728$ and~$0$.
\end{example}

 \begin{example}
   $\cym_2 = \bP(4, 10, 12, 16, 18, 22, 24, 28, 30, 36, 42)$ with $11$
   coordinates $t_{i_1i_2}$ of weight
   $\delta(i_1,i_2) = 42- \frac{42}{3}i_1 - \frac{42}{7}i_2 >0$ for
   $0\le i_1\le 1$ and $0\le i_2\le 5$.
   Note: condition $\delta(i_1,i_2)>0$ excludes the variable $t_{15}$
   of weight $-2$.  The family of hypersurfaces is
   \begin{displaymath}
      \Big\{ x_0^2 + x_1^3 + x_2^7 + \sum_{i_1,i_2} t_{i_1i_2} x_1^{i_1}x_2^{i_2}
      \cdot x_3^{\delta(i_1,i_2)} =0 \Big\} 
      \to \cym_2 
    \end{displaymath}
    It is a family of hypersurfaces in
    $\bP(\tfrac{42}2,\tfrac{42}3,\tfrac{42}7, 1) = \bP(21,14,6,1)$.
    
    The previous moduli space $\cym_1$ is embedded into $\cym_2$ by 
   \begin{displaymath}
     x_0^2 + x_1^3 + t_{1} x_1x_2^4 + t_{0}x_2^6 \longmapsto 
     x_0^2 + x_1^3 + t_{1} x_1 \ell^4 \cdot x_3^4 +
     t_{0} \ell^6 \cdot x_3^6 + \ell^7,
   \end{displaymath}
   where $\ell = x_2 - \frac17 t_{0} \cdot x_3^6$.
   Explicitly:
   \begin{eqnarray*}
     &(t_{10}, t_{11}, t_{12}, t_{13}, t_{14}) = 
     (\frac1{7^4} t_{0}^4t_{1}, -\frac4{7^3}t_{0}^3t_{1},
     \frac6{7^2}t_{0}^2t_{1}, -\frac47 t_{0}t_{1}, t_{1}) \\
     &(t_{00}, t_{01}, t_{02}, t_{03}, t_{04}, t_{05}) = 
     (\frac6{7^7} t_{0}^7, -\frac5{7^5}t_{0}^6,
     \frac{12}{7^4}t_{0}^5,
     -\frac{15}{7^3}t_{0}^4,  \frac{10}{7^2}t_{0}^3, -\frac37 t_{0}^2)
   \end{eqnarray*}
   The weights of the coordinates $t_{i_1,i_2}$ in $\cym_2$ are all even.
   This reflects the fact that the automorphism group of a generic
   fiber $(X,D)$ is $\mu_2$, defined by the formula
   $(x_0,x_1,x_2,x_3)\mapsto (x_0,x_1,x_2,-x_3)$,
   which on the affine patch $\bA^{3} = \{x_{3}=1\}$ acts by
   $(x_0,x_1,x_2) \to (-x_0,x_1,x_2)$.

   Each surface $[X]\in\cym_2$ has fixed $A_1A_2A_6$-singularities at
   infinity.  The surfaces $[X]\in\cym_2\setminus \partial\cym_2$ are
   K3 surfaces with isolated ADE singularities in the finite part
   $X\setminus D$, none for a general $X$.  The most singular surfaces
   $[X]\in\partial\cym_2$ have a single log canonical singularity in
   $X\setminus D$ which is resolved by inserting an irreducible
   divisor $E$ that is either a smooth elliptic curve or a rational
   curve with a node. In this way, varieties from $\cym_1$
   reappear in the description of varieties in $\partial \cym_2$.
 \end{example}

 \begin{example}
   For $n=3$, $\mu=1\cdot 2 \cdot 6\cdot 42 = 504$. There are $252$
   monomials $t_Ix^I$ with $\delta_I<0$ and $252$ with $\delta_I>0$. So
   $\cymm_3$ is a $251$-dimensional weighted projective space
   $\bP(4,10,12,\dotsc,1764,1806)$.
 \end{example}

\subsection{First properties of the fibers} 
\label{sec:first-properties}

In this section, we work over an algebraically closed field $k$, and
after the first lemma we assume that $\chr k \not|\, \prod_{k=0}^n s_k$.

\begin{lemma}[cf. \cite{esser2023varieties-of-general}, Lemma 3.10] 
  $\amb=\bP(a_0,\dotsc, a_n,1)$ is a Fano variety with canonical
  $\mu_{a_k}$-quotient singularities.
\end{lemma}
\begin{proof}
  $\amb$ is a toric variety, let $\Delta$ be its toric boundary. The
  pair $(\amb,\Delta)$ is log canonical by
  \cite[Lemma~3.1]{alexeev1996log-canonical-singularities} and admits
  a toric resolution of singularities; this is 
  true in any characteristic. Since $\Delta \in |-K_{\amb}|=|\cO(d)|$
  is an ample Cartier divisor by Lemma~\ref{lem:P-Gorenstein}, $\amb$ has canonical singularities and
  is a Fano variety. $\amb$ is covered by the affine charts
  $\bA^{n+1}/\mu_{a_k}$ and $\bA^{n+1}$. Thus, it has
  $\mu_{a_k}$-quotient singularities.
\end{proof}

We have already proved part of Theorem~\ref{mainthm:family} above: we
constructed the universal family $\pi_n$, proved part (1), found
$\omega_{\cX_n/\cym_n}$ and described the embedding $\iota_{n-1}$. We
now establish parts (2), (3) and (4) of Theorem~\ref{mainthm:family},
working over $\bZ\big[1/\prod_{k=0}^n s_k\big]$.

\begin{theorem}\label{thm:mainthm1-234} 
  Let $(X,D)$ be a geometric fiber of the family $\pi_n$. Then:
  \begin{enumerate}
    \setcounter{enumi}{1}
  \item $X\setminus D$ has only finitely many (none for a general $X$) isolated
    singularities.
  \item Along the infinite divisor $D$, $X$ has some
    $\mu_{\gcd(a_k,a_{k'})}$-quotient singularities coming from the
    singularities of $\amb$, the same for every fiber.
  \item For $n\ne0$, $X$ is irreducible, and for $n\ne 1$, $X$ is normal.
  \end{enumerate}
\end{theorem}
\begin{proof}
  For any
  $t=(t_I)$ with $t_I\in k$, in the one-parameter family over
  $\bA^1_\lambda$ given by the equation
\begin{displaymath}
  F(\lambda.t) =\sum_{k=0}^n x_k^{s_k} + \sum_I \big(\lambda^{\delta_I} t_I\big) x^I \cdot x_{n+1}^{\delta_I}
\end{displaymath}
the central fiber over $\lambda=0$ is isomorphic to $X_0$ and all the
fibers over $\bA^1\setminus 0$ are isomorphic to $X_t$. Thus, any
property of $X_0$ that is open in the Zariski topology also holds for
$X_t$.
Let $C(X_0), C(X_t)\subset \bA^{n+2}$ be the affine cones over
$X_0,X_t$. Let $p\in X_0$ be the point $(0,\dotsc, 0, 1)$.  Since
$\chr k$ does not divide $\prod_{k=0}^n s_k$, $C(X_0\setminus p)$ is
smooth outside of the line $k.p$. Then $C(X_t)$ is smooth outside of
finitely many lines $k.p_1,\dotsc, k.p_r$.  $X_t$ is also well-formed
(\cite[Def.~6.9]{fletcher2000working-with-weighted}).  Then the
singularities of $X_t\setminus \{p_1,\dotsc, p_r\}$ all come from the
singularities of the ambient space
(cf. \cite[Prop.~8]{dimca1986singularities}). By the previous lemma,
they are some $\mu_{a_k}$-quotient singularities. But since $X_t$ does
not contain the coordinate points, they are in fact at worst
$\mu_{\gcd(a_k, a_{k'})}$-quotient singularities for some $k\ne k'$.

For $n\ne0$, $X_0$ is irreducible, and for $n\ne 1$, it is normal. Both
of these are open properties, so they hold for $X_t$ as well.
\end{proof}

Let us also prove the easy part (1) of
Theorem~\ref{mainthm:singularities}.

\begin{lemma}\label{lem:canonical-near-bdry}
  Near $D$, the pair $(X,D)$ has log canonical singularities, and the
  variety $X$ has canonical singularities.
\end{lemma}
\begin{proof}
  Let $\Delta$ be the toric boundary of $\amb$ and
  $\Delta_X =\Delta\cap X$. Let $\wt\amb\to\amb$ be a toric resolution
  of singularities.  By the same argument as in the proof of the
  previous theorem, it provides a resolution of singularities of the
  pair $(X,\Delta_X)$, and the discrepancies of $(X,\Delta_X)$ equal
  those of $(\amb,\Delta)$. So there exists an open neighborhood $U$
  of $\Delta_X$ such that $(X,\Delta_X)$ is log canonical on $U$ and
  $U\setminus\Delta_X$ is smooth.  For every exceptional divisor $E$
  on a resolution of $X$ for the discrepancies one has
  $a(X,E) \ge a(X,\Delta_X,E) +
  \operatorname{ord}_E(\Delta_X)$. Since $\Delta_X$ is Cartier,
  for every $E$ with the image in $\Delta_X$ we get $a(X,E)\ge
  -1+1\ge0$. So $X$ is canonical on $U$.
\end{proof}

\presectionskip
\section{Singularities of nondegenerate hypersurfaces in $\cym_n$}
\label{sec:sings-nondegenerate}

\subsection{The standard theory}
\label{sec:newton-general}

Let $f(x_0, \dotsc, x_n)= \sum_I c_I x^I$ with
$I = (i_0, \dotsc, i_n) \in\bZ^{n+1}_{\ge0}$ be a polynomial defining
an isolated singularity $X=(f=0)\subset\bC^{n+1}$ at the origin.  Its
Newton polyhedron is 
\begin{displaymath}
  \newton(f) = \Conv\big(I + \bR^{n+1}_{\ge0} \mid c_I\ne 0\big),
  \qquad 
\end{displaymath}
Denote by $\newton^\circ(f)$ its interior.  One says that $f$ is
\emph{(Newton) nondegenerate} if for any compact face $\gamma$ of
$\newton(f)$ the leading term $f_\gamma$ defines a nonsingular
hypersurface in $(\bC^*)^{n+1}$.  For a linear function
$w(I)=w_0 i_0+\dotsc +w_n i_n$ on $\bR^{n+1}$ with
$w_k>0$, let
\begin{displaymath}\textstyle
  w(f) = \min\, \{w(I) \mid I\in\newton(f) \}.
\end{displaymath}
In particular, one has
$w(\prod_{k=0}^n x_k) = w(\bone) = \sum_{k=0}^n w_k$,
where $\bone=(1,1,\dotsc,1)$.  Define
\begin{displaymath}
  b_w = 1 + w(f) - \sum_{k=0}^n w_k
  \quad\text{and}\quad
  r_w = \frac{w(f)}{\sum_{k=0}^n w_k}.
\end{displaymath}
Clearly, $b_w\le 1$ (resp.\ $b_w<1$) iff $r_w\le 1$
(resp.\ $r_w<1$). But unlike $b_w$, $r_w$ is
homogeneous: $r_{\lambda w} = r_w$ for 
$\lambda\in\bR_{>0}$. Equally obvious is that
\begin{displaymath}
  \max_w r_w \le1, \text{ resp.\ } \max_w r_w< 1
  \iff
  \bone \in\newton(f), \text{ resp.\ } \bone\in\newton^\circ(f).
\end{displaymath}

The following statement is well known and goes back to
\cite[I.12.1]{arnold1985singularities-of-differentiable}, see e.g.\ 
\cite[\S8]{kollar1997singularities-of-pairs}, 
\cite{howald2001multiplier-ideals, howald2003multiplier-ideals}, 
\cite[Prop.~2.9]{ishii2001hypersurface-exceptional}.
We will need a similar computation in the proof of
Lemma~\ref{lem:blowup-discrepancy}, so for clarity we provide a (very
short) proof.

\begin{theorem}\label{thm:newton-lc}
  If $\bone\notin \Gamma_+(f)$ then $(f=0)$ is not log canonical at the origin. Now 
  suppose that $f$ is nondegenerate and $X=(f=0)$ has an isolated
  singularity at the origin. Then $X$ is log canonical
  (resp.\ canonical) at the origin $\iff$ $\max_w r_w \le1$
  (resp.\ $<1$) $\iff$ $\bone \in\newton(f)$
  (resp.\ $\bone\in\newton^\circ(f)$).
\end{theorem}

\begin{proof}
  Denote $Y=\bC^{n+1}$ and let $\phi\colon \wY\to Y$ be a proper
  birational toric morphism dominating the projective morphism
  $Y'\to Y$ corresponding to $\Gamma_+(f)$ and such that $\wY$ is
  smooth. Denote by $\wX$ the proper transform of $X$ and by
  $\varphi\colon\wX\to X$ the induced morphism.
  The fan $\fF_\wY$ of
  $\wY$ is a refinement of the fan $\fF_Y$ of $Y$ (the first quadrant),
  and the exceptional divisors $E_w$ of $\phi$ are in bijection
  with the new rays $w=(w_0,\dotsc,w_n)\in\bZ_{>0}^{n+1}$ in
  $\fF_\wY$.  Let $D= \sum_{k=0}^n D_k = (\prod_{k=0}x_k)$ be the
  toric boundary of $Y$. Then the toric boundary of $\wY$ is
  $\phi_*\inv(D) + \sum_{w\in V}E_w$. We compute
  \begin{eqnarray*}
    &&\phi^*\big(K_{Y} + D) = 0 = K_\wY + \phi_*\inv (D) +
       \sum_{w\in V} E_w\\
    &&\phi^*(X) = \wX + \sum  w(f) E_w,
       \qquad
       \phi^*(D) = \phi_*\inv(D) + \sum w(\bone) E_w.
  \end{eqnarray*}
  This combines to 
  \begin{equation}\label{eq:log-discrepancy}
    \phi^*(K_{Y} + X ) =
    K_{\wY} + \wX + \sum_{w\in V}b_w E_w,
      \quad
    \varphi^*(K_X) = K_\wX + \sum b_w E_w|_\wX.
  \end{equation}
  Assume that $f$ is nondegenerate. Then $\varphi$ is a resolution of
  singularities of $X$ and the divisor $\sum_w E_w|_{\wX}$ is a
  reduced normal crossing divisor on $\wX$. If $\bone \in\newton(f)$
  (resp.\ $\bone\in\newton^\circ(f)$) then $b_w\le 1$ (resp.\ $b_w<1$) and
  the last formula says that $X$ is log canonical (resp.\ klt). And
  since $X$ is Gorenstein, klt is equivalent to canonical.

  On the other hand, if some $b_w>1$ then already on the weighted
  blowup $\phi_w\colon Y_w\to Y$ corresponding to the star subdivision of
  $\fF_Y$ obtained by adding the ray $w$, the divisor $E_w|_{\hX_w}$
  on the normalization $\hX_w$ of the strict transform $X_w =
  \phi_w\inv(X)$ produces a divisor of log discrepancy $<-1$, even
  without assuming that $f$ is nondegenerate. So $X$ is not log canonical.
 \end{proof}

\begin{corollary}\label{cor:brieskorn-pham-nonlc}
  The Brieskorn-Pham singularity~\eqref{eq:brieskorn-pham} is not log
  canonical. 
\end{corollary}
\begin{proof}
  Indeed, the only compact face of $\newton(f_0)$ is given by the
  inequality $w(I)\ge1$ with
  $w = (\frac1{s_0}, \dotsc, \frac1{s_n})$. So $f_0$ is
  nondegenerate and $r_w = 1/\sum_{k=0}^n \frac1{s_k} >1$ by the
  Egyptian fraction identity~\eqref{eq:egyptian}.
\end{proof}

\subsection{Convexity lemma}
\label{sec:combinatorial-lemma}

Let $e_0,\dotsc, e_n$ be the standard basis of $\bR^{n+1}$.

\begin{lemma}\label{lem:combinatorial-lemma}
  Let $I=(i_0,\dotsc, i_n)\in \bZ_{\ge0}^{n+1}$ be a point with
  $\delta_I>0$, and $\sigma_I$ be a simplex with the vertices $s_ke_k$
  ($0\le k\le n$) and $I$. 
  Then  $\bone\in \sigma_I$, and if $\bone \notin \sigma_I^\circ$ then
  \begin{equation}
    \label{eq:lies-in-face}
    i_n + \sum_{k=0}^{n-1} a_{n-1,k} i_k = d_{n-1}.
  \end{equation}
\end{lemma}
\begin{proof}
  Let $\delta:=\delta_I\in\bZ_{>0}$ be the degree difference of $I$
  as in Equation~\eqref{eq:degree-diff}, and set
  $\displaystyle\alpha=\frac1{\delta}$ and
  $\beta_k = \displaystyle\frac{\delta-i_k}{\delta s_k}$ for
  $0\le k\le n$. One has
  \begin{displaymath}
    \alpha I + \sum_{k=0}^n \beta_k (s_ke_k)
    = \big( \dotsc, \frac{i_k}{\delta} + \frac{\delta-i_k}{\delta},
    \dotsc \big) =     \bone,
    \end{displaymath}
    \begin{displaymath}
    \alpha+\sum_{k=0}^n\beta_k
     =\frac1{\delta}+\sum_{k=0}^n\frac{\delta-i_k}{\delta s_k}
    =\sum_{k=0}^n\frac1{s_k}+\frac1d=1,
  \end{displaymath}
   using $\delta=d(1-\sum i_k/s_k)$ and the Egyptian
   identity~\eqref{eq:egyptian}.  So $\alpha$ and $\beta_k$ are the
   barycentric coordinates of the point $\bone$ in the coordinates
   $\{I, s_ke_k\}$. The condition $\bone\in\sigma_I$ is equivalent to
   $\alpha\ge0$ and $\beta_k\ge0$ for $0\le k\le n$. The first
   inequality is clear, and the second one means that $\delta\ge i_k$
   for all~$k$. Indeed, taking Equation~\eqref{eq:degree-diff} modulo
   $s_k$ and using the Congruence relation~\eqref{eq:congruence} gives
  \begin{displaymath}
    \delta \mymod{s_k} = -a_ki_k \mymod{s_k} = i_k \mymod{s_k}. 
  \end{displaymath}
  Since $i_k<s_k$, this implies that $\delta\ge i_k$ for all $k$ and
  proves that $\bone\in\sigma_I$. If $\bone$ lies on the boundary of
  $\sigma_I$ then one must have $\beta_k=0$ for some $k$,
  i.e.\ $\delta=i_k$. Then
  \begin{displaymath}
    i_k \mymod{s_n} = \delta\mymod{s_n} = i_n \mymod{s_n}.
  \end{displaymath}
  Since both $i_k$ and $i_n$ lie in the interval $[0, s_n)$, this
  means that $i_k=i_n$. So $\delta=i_n$. Recalling the
  definition~\eqref{eq:degree-diff} of $\delta$, this translates to
  \begin{displaymath}
    d - \sum_{k=0}^n a_k i_k = i_n \iff
    s_n i_n  + \sum_{k=0}^{n-1} a_k i_k = d 
  \end{displaymath}
  because $a_n+1 = d_{n-1}+1 = s_n$. Dividing this identity by
  $s_n$, we obtain~\eqref{eq:lies-in-face}.
\end{proof}

\subsection{Singularities of nondegenerate hypersurfaces in $\cym_n$}
\label{sec:consequences}

In the next lemma we work with a singularity
$f_t= \sum_{k=0}^n x_k^{s_k} + \sum_{I\in\cI_n^+} t_I x^I$ of the general form
\eqref{eq:F_t} in which the multi-indices $I$ go over the extended index set
\begin{equation}
  \label{eq:extended-index-set}
    \cI_n^+ = \{ (i_0, \dotsc, i_n) \mid 0\le i_k\le s_k-1,\ \delta_I > 0 \}.
\end{equation}

\begin{lemma}\label{lem:comes-from-previous}
  Assume that $(f_t=0)$ has an isolated nondegenerate singularity at
  the origin and that $f_t\ne f_0$, i.e.\ there is at least one
  monomial $t_Ix^I$ with $\delta_I>0$. Then this singularity is log
  canonical. Moreover, it is canonical unless 
  \begin{math}
    f_t = g_{t'} + x_n^{s_n},
  \end{math}
  where $g_{t'}(x_0,\dotsc, x_n)$ is a weighted homogeneous
  polynomial of degree $d_{n-1}=s_n-1$ if we set $\deg x_k = a_{n-1,k}$ for
  $0\le k\le n-1$ and $\deg x_n=1$, in which case it is strictly log canonical.
\end{lemma}
\begin{proof}
  The first half follows at once by Theorem~\ref{thm:newton-lc} and
  Lemma~\ref{lem:combinatorial-lemma} since $\sigma_I$ is a subset of
  $\Gamma_+(f_t)$.  For the ``moreover'' part, recall from
  Theorem~\ref{thm:newton-lc} that the singularity is canonical iff
  $\bone\in\Gamma_+(f)^\circ$. If there exists a monomial $t_Ix^I$
  with $\delta_I>0$ such that $\bone\in\sigma_I^\circ$, then
  $\bone\in\Gamma_+(f)^\circ$ since $\sigma_I^\circ$ is open and
  $\sigma_I\subset\Gamma_+(f)$, hence the singularity is canonical.
  Otherwise $\bone\notin\sigma_I^\circ$ for every monomial with
  $\delta_I>0$, and Lemma~\ref{lem:combinatorial-lemma} (see~\eqref{eq:lies-in-face}) gives
  $i_n+\sum_{k=0}^{n-1}a_{n-1,k}i_k=d_{n-1}$ for every such
  $I$. Equivalently, $f_t-x_n^{s_n}$ is weighted homogeneous of degree
  $d_{n-1}$ for weights $\deg x_k=a_{n-1,k}$ ($k\le n-1$),
  $\deg x_n=1$. In this case
  $\bone\in\Gamma_+(f)\setminus\Gamma_+(f)^\circ$, so the singularity
  is strictly log canonical.
\end{proof}

\begin{corollary}\label{cor:nondegenerate-lc-can}
  Let $[(X,D)]\in\cym_n$ and $p \in X\setminus D$ be a nondegenerate
  singularity. Then $X$ is log canonical at $p$. Moreover, it is
  canonical unless $[X]\in \iota_{n-1}(\cym_{n-1})$ as in
  Section~\ref{sec:En-1-En}, in which case it is strictly log canonical.
\end{corollary}
\begin{proof}
  Say, $p=(c_0,\dotsc, c_n)\in X\setminus D$ is a singular point in
  the affine chart $\bC^{n+1}=\{x_{n+1}=1\}$. After making a
  substitution $x_k \to x_k + c_k$ the singularity $p$ is moved to the
  origin, and the polynomial $f_t$ takes the form
  $\sum_{k=0}^n x_k^{s_k} + \sum t_I x^I$ in which the multi-indices
  go over the extended index set $\cI_n^+$ in \eqref{eq:extended-index-set}.
  By Lemma~\ref{lem:comes-from-previous}, $X$ is log
  canonical at $p$. Moreover, if it is not canonical then it is obtained from
  a homogeneous degree $d_{n-1}$ polynomial $g_{t'}$ by adding the
  $x_n^{s_n}$ term. Putting it back into the normal form by
  completing the $s_k$-th powers is done uniquely, and the result
  matches $\iota_{n-1}(g_{t'})$ defined in Section~\ref{sec:En-1-En}.
\end{proof}

\begin{remark}
  Since the proofs in this section depend only of $\Gamma_+(f)$, the
  conclusions of Corollary~\ref{cor:nondegenerate-lc-can} remain true
  if $f$, in addition to at least one nonzero monomial $t_Ix^I$ of
  degree $<d$, also contains monomials of degree $>d$.
\end{remark}

\presectionskip
\section{Varieties fibered in Calabi-Yau hypersurfaces}
\label{sec:fibered-varieties}

In this section, we work over $\bC$ because we use Inversion of
Adjunction Theorem~\ref{thm:inversion-of-adjunction} in the proofs of
Lemma~\ref{lem:prep} and Theorem~\ref{thm:lct-formula}.  Let $C$ be a
smooth curve and $\pi\colon X\to C$ be a family of hypersurfaces in
$\amb$ defined by a polynomial
\begin{equation}\label{eq:main-equation-again}
  F = \sum_{k=0}^n x_k^{s_k} + \sum_{I\in\cI_n^+} t_Ix^I \cdot x_{n+1}^{\delta_I},
  \quad I = (i_0, \dotsc, i_n).
\end{equation}
The multi-indices $I$ go over the extended index set $\cI_n^+$ defined
in \eqref{eq:extended-index-set}. 
As in Definition~\ref{def:family-hypersurfaces}, here
$t_I\in H^0(C,  L^{\delta_I})$ are sections for
a certain ``fundamental'' line bundle $L$ on $C$.
As always, we denote by $D$ the ``infinite'' divisor $(x_{n+1}=0)$.
We assume that at
least one of the coefficients $t_I$ does not vanish identically on
$C$. But, crucially, we allow all $t_I$ to vanish simultaneously at
some points $p\in C$. If that happens, the fiber $X_p = \pi\inv(p)$ is
``cuspidal'': it is isomorphic to the weighted cone
\begin{equation}
  \label{eq:weighted-cone}
  \Big(\sum_{k=0}^n x_k^{s_k} = 0\Big) \subset \amb.
\end{equation}

If $n=1$ then $X\to C$ is an elliptic surface with a section
$D$. We now generalize several classical results in the theory of
elliptic surfaces to the higher-dimensional case, such as:
\begin{enumerate}
\item The minimal form and a criterion for
  $X$ to have $ADE$ (i.e.\ canonical) singularities \cite[Section
  III.3]{miranda1989the-basic-theory}.
\item The log canonical threshold of the pair $(X,X_p)$ and the 
  canonical class formula \cite{ueno1973classification-of-algebraic,
    fujita1986zariski-decomposition, fujino2000canonical-bundle}.
\item Tate's algorithm for determining the Kodaira type of a fiber,
  with Dokchitser's improvements \cite{dokchitser2013remark-on-tate}. 
\end{enumerate}

\subsection{The minimal form}
\label{sec:minimal-form}

Let $(C,p)$ be a smooth pointed curve, and let $\tau$ be a local
parameter at $p$.

\begin{definition}\label{def:min-reduced-valuation}
  For each $I$, let $\val_p t_I = \val_\tau t_I$ be the multiplicity of
  $\tau$ in $t_I$. If $t_I\equiv 0$, we set $\val_p t_I=+\infty$. We
  define the \emph{reduced valuation} to be
  \begin{displaymath}
    \ovval_p\, t_I = \frac{\val_p t_I}{\delta_I}
    \quad\text{with } \delta_I \text{ defined as in } \eqref{eq:degree-diff},
  \end{displaymath}
  and $\nuu_p:= \min_I \ovval_p\, t_I$. We say that a family
  \eqref{eq:main-equation-again} is \emph{in minimal form} if $\nuu_p<1$.
\end{definition}

The number $\nuu_p$ measures the largest common factor $\tau^{\nuu_p\delta_I}$
``hidden'' in $t_I$.

\begin{lemma}\label{lem:diagram-Ys}
  For any positive integer $m$, there exists a commutative diagram
  \begin{center}
  \begin{tikzcd}
    & \wY 
    \arrow[ld, "\phi"'] \arrow[rd, "\phi^{(m)}"]
    \\
    Y && Y^{(m)}
  \end{tikzcd}
\end{center}
in which
\begin{enumerate}
\item $V$, $V^{(m)}$ are graded vector bundles over $C$ defined in
  Section~\ref{sec:bPV} with fundamental line bundles $L$ and $L^{(m)}$
  on $C$, and $Y=\bP(V)$, $Y^{(m)}=\bP(V^{(m)})$ are the corresponding
  locally trivial $\bP$-fibrations over
  $C$. 
\item $\phi$ and $\phi^{(m)}$ are normalized weighted blowups with
  exceptional divisors resp.\ $\bE\simeq Y_p^{(m)}$ and~$\bE^{(m)}$. $\phi$
  contracts $\bE$ to the point $0\in\bA^{n+1}\subset Y_p$, and
  the map $\phi^{(m)}|_{\bE^{(m)}}$ is generically a
  $\bP^1$-fibration over the infinite divisor $\bD_p^{(m)}\subset
  Y^{(m)}_p$.
\item The fiber $\wY_p$ is $\bE\cup \bE^{(m)}$. It is
  reduced and generically has double crossing singularities along
  $\bD_p = \bE\cap\bE^{(m)} \simeq \bP(a_0,\dotsc,a_n)$. The ambient
  variety $\wY$ generically has $A_{m-1}$-singularities along
  $\bD_p$.
\item $\wY$ and $\wY_p$ are Cohen-Macaulay.
\item The divisor $md\bE$ on $\wY$ is Cartier.
\item For the fundamental line bundles, one has $L^{(m)}=L(-mp)$.
\end{enumerate}
\end{lemma}
\begin{proof}
  We work locally in a neighborhood of $p\in C$, so we could assume that
  $Y\simeq \amb\times C$ and $Y^{(m)}\simeq \amb\times C$ if we wished.  We
  build the above diagram using a toroidal construction, which we
  describe by fans in the lattice $N = \bZ^{n+1}\times\bZ$. We begin
  with the fan $\fF_\amb$ of the weighted projective space
  $\amb=\bP(a_0,\dotsc, a_n,1)$. This is a simplicial fan in
  $\bZ^{n+1}$ whose rays are the basis vectors $e_k$ for $0\le k\le n$
  and the vector $v = -\sum_{k=0}^n a_k e_k$. The cones are
  generated by the subsets of cardinality $\le n+1$ of these $n+2$
  rays.  The fan of $Y$ is the Cartesian product
  $\fF_Y = \fF_\amb \times \bR_{\ge0}  e_{n+1}$, with $e_{n+1}$
  corresponding to the local parameter $\tau$ on $C$. The
  $(n+2)$-dimensional cone $\sigma=\sum_{k=0}^{n+1}\ \bR_{\ge0} e_k$
  corresponds to the open subspace $\bA^{n+1}\times C\subset Y$.

  We define $\phi\colon\wY\to Y$ to be the weighted blowup obtained by
  inserting the ray $w = \sum_{k=0}^n ma_k e_k + e_{n+1}$ into the
  interior of $\sigma$ and star-subdividing it.  Then the center of
  $\phi$ is the point $\{x_0=\dotsb= x_n=\tau=0\}$ in
  $\bA^{n+1} \subset Y_p $. The exceptional divisor $\bE$ of $\phi$ is
  $\bP(ma_0,\dotsc, ma_n,1)\simeq \bP(a_0,\dotsc,a_n,1)= \amb$ by
  the standard well-formed reduction~\cite[Prop.\
  1.3]{delorme1975espaces-projectifs}. 
  We define the contraction $\phi^{(m)}$ by forgetting the ray $e_{n+1}$
  in the fan $\fF_{\wY}$. The fan $\fF_{Y^{(m)}}$ is then also a
  Cartesian product $\fF_\amb\times\bR_{\ge0}w$, so $Y^{(m)}$ is a
  locally trivial $\amb$-fibration over $C$. The relation
  $e_{n+1} = mv + w$ shows that $\phi^{(m)}$ is a weighted blowup with
  the center $\{x_{n+1}=\tau=0\} = \bD_p\subset\amb\simeq Y^{(m)}_p$.

  The divisor $\bE^{(m)}$ is a toric variety corresponding to the fan
  $\operatorname{Star}(e_{n+1}) / \bR  e_{n+1}$ with the lattice
  $N / \bZ e_{n+1}$. This lattice is generated by the images
  $\bar e_k$, $\bar v$, $\bar w$, and the integral generators of the
  rays of $\fF_{\bE^{(m)}}$ are $\bar e_k$, $\bar v$ and $-\bar v$
  because $\bar w = -m\bar v$.  The map $\bE^{(m)}\to Y^{(m)}$
  corresponds to a projection $\fF_{\bE^{(m)}} \to \fF_{Y^{(m)}}$ of fans sending
  $\pm \bar v\to 0$. So it is generically a $\bP^1$-fibration whose
  image is the toric variety for the fan in
  $\bZ^{n+1}/(\bZ e_{n+1} + \bZ w)$ generated by $\bar e_k$,
  $0\le k\le n$ with a single relation $\sum a_k \bar e_k=0$. This is
  the fan of the infinite divisor $\bP(a_0,\dots,a_n)=\bD_p\subset\amb$.
  
  The fiber $\wY_p=\bE\cup\bE^{(m)}$ is reduced because both $e_{n+1}$
  and $w$ map to $1\in\bZ$ under the second projection
  $\bZ^{n+1}\times\bZ\to\bZ$. As it is true for any toroidal variety
  with a fan with two rays sharing a two-dimensional cone, the
  intersection $\bE\cap\bE^{(m)}$ is generically a double crossing in
  $\bE\cup\bE^{(m)}$.  It is straightforward to check that
  $\fF_{\bE\cap\bE^{(m)}} = \fF_{\bD_p}$ for
  $\bD_p\subset\amb\simeq \bE$. The vectors $e_{n+1},w$ generate a
  sublattice of index $m$ in $N$, which implies that $\wY$ generically
  has Gorenstein index-$m$ cyclic quotient (i.e.\ $A_{m-1}$) singularities
  along $\bE\cap\bE^{(m)}$.

  Like any toroidal variety, $\wY$ is Cohen-Macaulay, and so is the
  Cartier divisor $\wY_p$ on it.  The divisor $md\bE$ is Cartier because
  the ray generators of every maximal cone span a sublattice with
  cotorsion annihilated by $md$.

  Let $\bD$ and $\bD^{(m)}$ be the families of infinite divisors in
  $Y\to C$ and $Y^{(m)}\to C$ respectively. The $\bQ$-divisor
  $\bD^{(m)}$ corresponds to a PL function $\gamma$ on the fan
  $\fF_{Y^{(m)}}$ which equals $1$ on $v$ and zero on the other
  rays. The pullback $\phi^{(m)}{}^*(\gamma)$ corresponds to the
  pullback PL function on $\fF_{\wY}$. The relation $e_{n+1}=mv+w$
  implies $\gamma(e_{n+1}) = m\gamma(v) + \gamma(w)=m$, which shows
  that
  $\phi^{(m)}{}^*(\bD^{(m)}) = (\phi^{(m)}_*)\inv(\bD^{(m)}) + m\bE^{(m)}$
   and $\phi_*\phi^{(m)}{}^*(\bD^{(m)}) = \bD + mY_p$.
  Since $\bD$ and $\bD^{(m)}$ are sections of $\cO_{\bP(V)}(1)$ and
  $\cO_{\bP(V^{(m)})}(1)$ respectively, this implies that $L = L^{(m)}(mp)$,
  proving (6).
\end{proof}

\begin{remark}
  The two components $\bE$ and $\bE^{(m)}$ of the fiber $\wY_p$ have
  complementary singularities along $\bD_p$. The first one is covered by
  the quotients $\frac{(-1,\ a_k, k\ne j)}{d}$ and the second one by
  $\frac{(1,\ a_k, k\ne j)}{d}$. The latter are canonical, while the
  former are not.
\end{remark}

\begin{theorem}\label{thm:minimal-form}
  Suppose that a family~\eqref{eq:main-equation-again} is not in minimal
  form, and let $m\le \nuu_p$ be a positive integer. Then there
  exists a commutative diagram
  \begin{center}
  \begin{tikzcd}
    & \wX 
    \arrow[ld, "\varphi"'] \arrow[rd, "\varphi^{(m)}"]
    \\
    X && X^{(m)}
  \end{tikzcd}
  \end{center}
  in which
  \begin{enumerate}
  \item $X$ and $X^{(m)}$ are families as in
    \eqref{eq:main-equation-again}, with $t^{(m)}_I =
    t_I/\tau^{m\delta_I}$, $\nuu^{(m)}_p = \nuu_p-m$.
  \item $\varphi$ and $\varphi^{(m)}$ are normalized weighted blowups
    with exceptional divisors resp.\ $E\simeq X_p^{(m)}$
    and~$E^{(m)}$. $\varphi$ contracts $E$ to a point of
    $X_p$, and the map $\varphi^{(m)}|_{E^{(m)}}$ is generically a
    $\bP^1$-fibration over the infinite divisor $D_p^{(m)}\subset~X_p^{(m)}$.
\item The fiber $\wX_p$ is $E\cup E^{(m)}$. It is
  reduced and generically has double crossing singularities along
  $D_p = E\cap E^{(m)}$, the infinite divisor on $E$. The ambient
  variety $\wX$ generically has $A_{m-1}$-singularities along
  $D_p$.
\item $\wX$ and $\wX_p$ are Cohen-Macaulay.
  \end{enumerate}
\end{theorem}
\begin{proof}
  Again, we can work locally, over a small neighborhood of $p$.  Since
  $\val_p t_I \ge \nuu_p \delta_I\ge m\delta_I$ for all $I$, the sections
  $t^{(m)}_I$ are regular. Then $\ovval_p\,t_I^{(m)} = \ovval_p\,t_I -
  m$, so $\nuu^{(m)}_p=\nuu_p-m$.
  Let $X\subset Y$ and $X^{(m)}\subset Y^{(m)}$ be hypersurfaces in the weighted
  projective fibrations of Lemma~\ref{lem:diagram-Ys}.  Let $z$ and
  $z^{(m)}$ be local parameters at generic points of $\bE$ and $\bE^{(m)}$, so
  that $\tau = zz^{(m)}$. In coordinates, the blowups $\phi$ and $\phi^{(m)}$ are
  respectively
  \begin{eqnarray*}
    &x_k = z^{ma_k} u_k \quad\text{for}\quad 0\le k\le n,
       &x_{n+1} = u_{n+1} \\
    &x^{(m)}_k = u_k \quad\text{for}\quad 0\le k\le n,
       &x^{(m)}_{n+1} = z^{(m)}{}^m u_{n+1} 
  \end{eqnarray*}
  Then $\sum_{k=0}^n x_k^{s_k} = z^{md} \sum_{k=0}^n u_k^{s_k}$ and 
  \begin{displaymath}
    t_I x^I x_{n+1}^{\delta_I} =
    (\tau^{m\delta_I} t^{(m)}_I) (z^{\sum ma_k i_k} u^I) u_{n+1}^{\delta_I}
    =
    z^{md} t^{(m)}_I u^I (z^{(m)}{}^{m} u_{n+1})^{\delta_I},
  \end{displaymath}
  This implies that the multiplicity $w(f)$ of $F$ at the origin along
  the divisor $\bE$ is $md$. It also implies the identity
  \begin{displaymath}
    F = \sum_{k=0}^n x_k^{s_k} + \sum_I t_I x^I \cdot x_{n+1}^{\delta_I} =
    z^{md} \Big(
    \sum_{k=0}^n x_k^{(m)}{}^{s_k} + \sum_I t^{(m)}_I x^{(m)}{}^I
    \cdot x_{n+1}^{(m)}{}^{\delta_I}\Big) = z^{md}F^{(m)}.
  \end{displaymath}
  This proves that $\phi^*(X) = \phi^{(m)}{}^*(X^{(m)}) + md\bE$. We
  set $\wX\subset\wY$ to be $\phi^{(m)}{}^*(X^{(m)})$, which coincides
  with the strict transform $\phi_*\inv(X)$.
  Since $\wX$ is a Cartier divisor in $\wY$, it is Cohen-Macaulay and
  so is its fiber over $p$. Parts (2) and (3) follow immediately from
  parts (2) and (3) of Lemma~\ref{lem:diagram-Ys}.
\end{proof}

\begin{corollary}\label{cor:minimal-form}
  For $m = \lfloor \nuu_p \rfloor$, the family $X^{(m)}\to C$ is in
  minimal form over $p$.
\end{corollary}

\begin{lemma}\label{lem:blowup-discrepancy}
  In the previous theorem, one has
  \begin{math}
    K_\wX = \varphi^*(K_X) - m E.
  \end{math}
\end{lemma}
\begin{proof}
  The same computation as in the proof of Theorem~\ref{thm:newton-lc},  
  Equation~\eqref{eq:log-discrepancy} gives
  \begin{displaymath}
    \varphi^*(K_X) = K_\wX + bE,
    \quad\text{where } b=1 + w(f)-\sum_{k=0}^n w(x_k) -w(\tau).
  \end{displaymath}
  As in the proof of the above theorem, $w(f)=md$, so
  \begin{math}
    b = 1 + md-\sum_{k=0}^n ma_k -1,
  \end{math}
  which equals $m$ by the Egyptian identity~\eqref{eq:egyptian}.
\end{proof}


\subsection{Log canonical threshold of a fiber}
\label{sec:lct}

The proofs below assume $n\ge2$ for simplicity. The case $n=1$, which
in any case is classical, is done in the same way by replacing ``lc''
by ``slc'' and the lc threshold by the slc threshold. (The only
difference is that a nodal curve $[X]\in\cym_1$ has an slc singularity.)

\begin{lemma}\label{lem:prep}
  Suppose that Theorem~\ref{mainthm:singularities} holds in
  dimension $n$. Then
  \begin{enumerate}
  \item For any family~\eqref{eq:main-equation-again} with $\nuu_p=0$,
    the $(n+1)$-dimensional pair $(X,X_p+D)$ is log canonical in a
    neighborhood of $X_p$.
  \item In the diagram of Theorem~\ref{thm:minimal-form}, if
    $\nuu_p^{(m)}=0$ then $\wX_p$ is semi log canonical and the pair
    $(\wX,\wX_p)$ is log canonical in a neighborhood of $\wX_p$.
  \end{enumerate}
\end{lemma}
\begin{proof}
  (1) Since $\nuu_p=0$, after putting the fiber $(X,D_p)$ into a normal form it
  gives a point of $\cym_n$, hence it is log canonical by the
  dimension-$n$ hypothesis. Inversion of
  Adjunction~\ref{thm:inversion-of-adjunction} gives log
  canonicity of $(X,X_p+D)$ near $X_p$.

  (2) Since the pair $(E,D_p)$ is isomorphic to a fiber of the family
  $(X^{(m)},D^{(m)})\to C$ with $\nuu_p^{(m)}=0$,
  after putting it into normal form
  it defines a point of $\cym_n$, hence $(E,D_p)$ is log canonical by
  the dimension-$n$ hypothesis (as in the proof of (1)).
  We claim that the pair $(E^{(m)},D_p)$ is log canonical as
  well. Indeed, changing the lattice $\la \bar e_k,\ 0\le k\le n\ra$
  of the fan $\fF_{\bE^{(m)}}$ to the sublattice
  $\la a_k\bar e_k,\ 0\le k\le n\ra$, we see that $\bE^{(m)}$ is a
  quotient of the locally trivial $\bP^1$-fibration
  $\wt{\bP}:=\Bl_{\mathrm pt}\bP^n \to \bP^{n-1}$ by the group
  $G=\prod_{k=0}^n \mu_{a_k}$. The variables $x_k$ become variables
  $y_k^{a_k}$ on $\wt{\bP}$ and $x_k^{s_k}$ becomes $y_k^d$, so
  $E^{(m)}$ is a $G$-quotient of the smooth subvariety
  $(\sum_{k=0}^n y_k^d = 0)\subset\wt{\bP}$ together with an infinite
  section. Then the pair $(E^{(m)},D_p)$ is log canonical by
  Corollary~\ref{cor:quotient-of-klt}.

  By Theorem~\ref{thm:minimal-form}, the variety $\wX_p=E\cup E^{(m)}$ is
  reduced, Cohen-Macaulay and has generically double crossing
  singularities along the nonnormal locus $D_p=E\cap E^{(m)}$. By the
  above, its normalization $(E,D_p) \sqcup (E^{(m)},D_p)$ is log
  canonical. So $\wX_p$ is semi log canonical. $\wX$ is normal since
  it is Cohen-Macaulay and $\codim\operatorname{Sing}\wX\ge2$. Then by
  the Inversion of Adjunction
  Theorem~\ref{thm:inversion-of-adjunction} the pair $(\wX,\wX_p)$ is
  log canonical in a neighborhood of $\wX_p$.
\end{proof}

\begin{theorem}\label{thm:lct-formula}
  Suppose that Theorem~\ref{mainthm:singularities} holds in dimension
  $n$. Then for the $(n+1)$-dimensional pair $(X,X_p)$ in a
  neighborhood of $X_p$ one has
  $\lct(X,X_p)=1-\nuu_p$.   
\end{theorem}
\begin{proof}
  Since the codimension of the singular locus of $X$ is $\ge2$ by
  Theorem~\ref{mainthm:family} and $X$ is Gorenstein, $X$ is normal.  
  Further, $X$ is canonical near the boundary $D$ by considering a
  canonical resolution of the ambient space $\amb$ as in the proof of
  Lemma~\ref{lem:canonical-near-bdry}. 
  So it suffices to work in the open part $X\setminus D$.
  If $\nuu_p=0$ then we are done by Lemma~\ref{lem:prep}.
  If $\nuu_p>0$, we make a cyclic base change to clear denominators so
  that $\nuu_p$ becomes an integer, run the minimal-form weighted
  blowup, and read off the lc threshold from the coefficient of the
  exceptional divisor.
  
  Let $r$ be a positive integer such that $r\nuu_p\in\bZ$ and let
  $(C',p')\to (C,p)$ be a degree-$r$ base change defined by the formula
  $\tau=\tau'{}^r$. The family
  $X'=X\times_C C'\to C'$ is of the same
  form~\eqref{eq:main-equation-again} but
  with coefficients
  $t'_I$ satisfying $\val_{p'} t'_I = r\val_p t_I$ and
  $\nuu_{p'} = r\nuu_p$. Take $m=\nuu_{p'}$ and consider the diagram
  $X'\gets \wX' \to X'{}^{(m)}$ of Theorem~\ref{thm:minimal-form}. One
  has $\nuu_{p'}^{(m)} = 0$, so by Lemma~\ref{lem:prep} the pair
  $(\wX', \wX'_{p'})$ is log canonical. 
  For $g\colon X'\to X$, we have $g^*(X_p)= rX'_{p'}$ and
  $g^*(K_X) = K_{X'} + (1-r)X'_{p'}$. Combined with
  Lemma~\ref{lem:blowup-discrepancy}, for any $c\in\bR$ this gives
  \begin{displaymath}
    (g\circ\varphi)^* (K_X + cX_p) = K_{\wX'}  
    + (1-r+rc)E'{}^{(m)}
    + (1-r+rc + r\nuu_p) E'
  \end{displaymath}
  By Theorem~\ref{thm:lc-under-maps}, the pair $(X,cX_p)$ is log
  canonical iff the pair in the right-hand-side of this equality is
  log canonical. 
  The coefficient of $E'$ in this expression is $\le 1$ iff
  $c\le 1-\nuu_p$. 
  On the other hand,
  setting $c=1-\nuu_p$, the formula above becomes
  \begin{equation}\label{eq:pullback-ineq}
    (g\circ\varphi)^* (K_X + (1-\nuu_p)X_p) = 
    K_{\wX'} + (1-r\nuu_p) E'{}^{(m)} + E'
    \le K_{\wX'} + \wX'_{p'},
  \end{equation}
  By Theorem~\ref{thm:lc-under-maps} and monotonicity of log canonicity,
  the pairs $\big(\wX', (1-r\nuu_p) E'{}^{(m)} + E'\big)$ and
  $(X,(1-\nuu_p)X_p)$ are log canonical. They are maximally log
  canonical because the coefficient of $E'$ is $1$.
  Thus, 
  $\lct(X,X_p)=1-\nuu_p$.
\end{proof}

\begin{corollary}\label{cor:family-lc}
  With assumptions as in Theorem~\ref{thm:lct-formula}, $X$ is log
  canonical in a neighborhood of $X_p$ iff $\nuu_p\le1$.
\end{corollary}
\begin{proof}
  By Theorem~\ref{thm:lct-formula}, $\lct(X,X_p)=1-\nuu_p$, so
  $\lct(X,X_p)\ge0$ iff $\nuu_p\le1$. Since $X_p$ is effective, log
  canonicity is monotone in the coefficient of $X_p$, hence
  $\lct(X,X_p)\ge0$ iff $(X,0)$ is log canonical, i.e.\ $X$ is log canonical.
\end{proof}

\begin{corollary}\label{cor:family-can}
  With assumptions as in Theorem~\ref{thm:lct-formula}, $X$ is
  canonical in a neighborhood of $X_p$ iff $\nuu_p<1$ and a general
  fiber of $X\to C$ is canonical.
\end{corollary}
\begin{proof}
Since $X$ is Gorenstein, it is canonical iff it is klt. 
A general fiber is canonical iff the centers of nonklt 
singularities of $X$ in a neighborhood of $X_p$ are contained
in $X_p$. And in this case $X$ is klt along $X_p$ iff
$\lct(X,X_p)=1-\nuu_p>0$, i.e.\ $\nuu_p<1$.
\end{proof}

\begin{proposition}\label{prop:implies-can}
  With assumptions as in Theorem~\ref{thm:lct-formula}, let
  $\{p_1,\dotsc, p_N\}$ be the points in $C$ over which the fibers are
  cuspidal, i.e.\ with $\nuu_{p}>0$, and let
  $U = C \setminus \{p_1, \dotsc, p_N\}$. Suppose that all
  $\nuu_{p}<1$ and that the restriction of
  $t_{0,\dotsc,0}\in H^0(C, L^{d_n})$ to $U$ is not equal to $d_{n-1}u^n$
  for a section $u \in H^0(U, M^{d_{n-1}})$ of a line bundle $M$ on $U$
  with $M^{s_n}\simeq L_U$. Then $X$ is canonical.
\end{proposition}
\begin{proof}
  By Corollary~\ref{cor:family-can} it suffices to check that not all
  fibers of $X\to C$ over $U$ are strictly log canonical. If they are
  then by Theorem~\ref{mainthm:singularities} in dimension $n$ the
  family $X_U\to U$ comes from a map $U\to \cym_{n-1}$. But then
  $t_{0,\dotsc,0}=d_{n-1}u^{s_n}$ by Remark~\ref{rem:constant-term}.
\end{proof}

\subsection{An extension of Kodaira's classification of singular fibers of
  elliptic surfaces to higher dimensions} 
\label{sec:kodaira}

\begin{remark} \label{rem:kodaira} Kodaira
  \cite{kodaira1963on-compact-III} classified fibers of relatively
  minimal elliptic surfaces $\wX\to C$. Assuming that $\wX$ has a
  section, one can equivalently talk of the Weierstrass fibrations
  $X\to C$ with canonical singularities, of the form
  \eqref{eq:main-equation-again} with $n=1$.  In Kodaira's notation
  $I_k$, $I^*_k$, $II$, \dots $IV^*$, the main part of the symbol --
  $I$, $I^*$, $II$, \dots $IV^*$ -- depends only on
  $\nuu_p = \min \ovval_p t_{i_0i_1}$ from
  Definition~\ref{def:min-reduced-valuation}. This can be easily seen
  from \cite[Table~IV.3.1]{miranda1989the-basic-theory}, see also
  \cite{dokchitser2013remark-on-tate}. The subscript $n$ depends on
  the valuation of the discriminant divisor at the point
  $p=(\tau=0)$. By analogy, we make the following definition for
  higher dimensions.
\end{remark}

\begin{definition}\label{def:kodaira}
  Consider a family $\pi\colon X\to C$ as in \eqref{eq:main-equation-again}
  with a fiber $X_p=\pi\inv(p)$, such that $X$ is
  canonical near $X_p$. We define the \emph{type} of $X_p$ to be the
  triple
  \begin{displaymath}
    \big( \nuu_p = \min_I \ovval_p\, t_I,\ 
    \operatorname{level}(j_{n}(p)),\
    \val_p(\Delta)\big),
  \end{displaymath}
  where $\Delta$ is the discriminant, and level was defined in
  Definition~\ref{def:level}. Here, the morphism $j_n\colon C\to \cymm_n$ is the
  extension of the map
  $C\setminus\text{\{cuspidal fibers\}} \to \cymm_n$ which exists by the
  valuative criterion of properness.
\end{definition}

\subsection{Canonical class formula}
\label{sec:canonical-formula}

Now, let $C$ be a smooth curve, not necessarily local, and
$\pi\colon X\to C$ be a family as in the previous section, meaning
that there are finitely many points $p_j \in C$ over which the fibers
are cuspidal cones~\eqref{eq:weighted-cone} 
(and, as $n$-folds, are not log canonical), with all the
other fibers being in $\cym_{n}$.
Assume that a general fiber of $\pi$ has canonical singularities.
We have an induced map
$C\setminus \{p_j\} \to \cymm_{n}$, and since the latter is proper,
also a map $j_{n}\colon C\to \cymm_{n}$. (Warning: it is a map to the
coarse moduli space $\cymm_n$ of the stack $\cym_{n}$; to get a map to the stack
$\cym_{n}$ one has to make a base change $C'\to C$ that kills the
monodromy, as in Section~\ref{sec:lct}.)

\begin{theorem}\label{thm:can-class-formula} 
  There is a natural equality of $\bQ$-line bundles
  \begin{equation}\label{eq:can-class-general}
    K_{X/C} = \pi^* \Big(
    \sum_{p\in C} \nuu_p p + j_{n}^* \cO_{\cymm_{n}}(1)
  \Big),
\end{equation}
where $\cO_{\cymm_{n}}(1)$ is the $\bQ$-line bundle on $\cymm_n$
associated to the line bundle $\cO_{\cym_n}(1)$.
\end{theorem}
\begin{proof}
  By Corollary~\ref{cor:minimal-form} and
  Lemma~\ref{lem:diagram-Ys}(6) we may assume that $\pi$ is in
  minimal form. Then $X$ is canonical by
  Corollary~\ref{cor:family-can} and
  Theorem~\ref{mainthm:singularities} proved in the next section. By
  Theorem~\ref{thm:lct-formula} one has $\nuu_p = 1-\lct(X,X_p)$.  A
  formula of the shape~\eqref{eq:can-class-general} exists for $\pi$
  by \cite{fujino2000canonical-bundle}, with some $\bQ$-line bundle
  $L$ called the ``moduli part'' in place of
  $j_{n}^* \cO_{\cymm_{n}}(1)$ in \eqref{eq:can-class-general}. We
  claim that $L = j_{n}^* \cO_{\cymm_{n}}(1)$.  By
  \cite{fujino2000canonical-bundle}, $L$ can be computed after a base
  change after which the family admits a semistable model. The
  statement now follows by Lemmas~\ref{lem:basic-family-canclass}
  and~\ref{lem:map-to-stack}. 
\end{proof}

\begin{remark}
  Working locally, let $X'\to (C,p)$ be obtained by deforming the
  coefficients $(t_I)$ defining $X$ to $(t'_I)$ so that the fibers of
  $X'$ are in $\cym_{n}$, i.e.\ the normalized minimal valuations are
  $\nuu_p'=0$. For $X'$, the formula \eqref{eq:can-class-general}
  holds by Lemmas~\ref{lem:basic-family-canclass}
  and~\ref{lem:map-to-stack}. How do the contributions in this formula
  change when one specializes from $X'$ to~$X$? The degree of the
  first term, increases by~$\nuu_p$.  On the other hand, the degree of
  the second term, with $j_{n}^* \cO_{\cymm_{n}}(1)$, drops by
  $\nuu_p$ because $t_I^{1/\delta_I}$ have a common multiple
  $\tau^{\nuu_p}$, which needs to be factored out and canceled to
  define, after a finite base change, a regular map to the weighted
  projective space $\cymm_{n}$. So the two contributions to
  \eqref{eq:can-class-general} cancel each other. This argument can be
  made into a second proof of Theorem~\ref{thm:can-class-formula}.
\end{remark}

\presectionskip
\section{Proof of (log) canonicity in general}
\label{sec:sing-general}

As an application of the results of the previous section, we now prove
parts (2) and (3) of Theorem~\ref{mainthm:singularities}
unconditionally, without assuming that the singularities are
nondegenerate.

\subsection{Proof of Theorem~\ref{mainthm:singularities}}
\label{sec:proof-of-thm2} 

Recall that part (1) was proved in Section~\ref{sec:first-properties}.

\begin{proof}[Proof of parts (2,3) of
  Theorem~\ref{mainthm:singularities}]
  Let $[X]\in\cym_n$.  We already know from
  Lemma~\ref{lem:canonical-near-bdry} that $X$ is canonical near $D$,
  so it suffices to prove the (log)canonicity of the affine part
  $X\setminus D$ defined by Equation~\eqref{eq:f_t}. 
  Rewrite this
  equation in the form
  \begin{equation}\label{eq:f_t-as-family}
    f_t = x_0^{s_0} + \dotsb + x_{n-1}^{s_{n-1}} +
    \sum t_{i_0\dotsb i_{n-1}}(x_n)\, x_0^{i_0}\dotsb x_{n-1}^{i_{n-1}}
  \end{equation}
  for some polynomials $t_{i_0\dotsb i_{n-1}}(x_n) \in \bC[x_n]$.
  This presents $X\setminus D$ as a fibration over $\bA^1$ into affine
  varieties $X'\setminus D'$ of dimension $n-1$.
  By induction, we can
  assume that Theorem~\ref{mainthm:singularities} holds in dimension
  $n-1$.

  Thus, the log canonicity of $X$ would follow from
  Corollary~\ref{cor:family-lc} if we could show that for every
  $c\in \bA^1$ one has $\nuu_c\le 1$.  Similarly, the canonicity of
  $X$ would follow from Proposition~\ref{prop:implies-can} if we could
  show the strict inequalities $\nuu_c< 1$. Indeed, the constant term
  in $f_t$ is
  \begin{math}
    t_{0,\dotsc,0}(x_n) = x_n^{s_n} + \text{(lower terms in $x_n$)}.
  \end{math}
  Since $s_n$ is not divisible by $s_{n-1}$, obviously
  $t_{0,\dotsc,0}(x_n) \ne d_{n-2} u^{s_{n-1}}$ for a rational function
  $u(x_n)$.

  So we have to establish that for any $\ell=x_n-c$ there exists
  $I' = (i_0, \dotsc, i_{n-1})\in\cI_{n-1}$ such that the polynomial
  $t_{I'}(x_n)$ satisfies the inequality
  $\val_\ell t_{I'}(x_n) \le \delta_{I'}$ (resp.\ the strict
  inequality).  After replacing $x_n$ by $\ell+c$,
  Equation~\eqref{eq:f_t} becomes
  \begin{equation}\label{eq:f_t-as-family-with-ell}
    f_t = \sum_{k=0}^{n-1}x_k^{s_k} + \ell^{s_n} +
    \sum_{(I',i_n)} t'_{I',i_n}\, x^{I'} \ell^{i_n}, \qquad I'=(i_0, \dotsc, i_{n-1})
  \end{equation}
  with some new constants $t'_{I',i_n}$.  In this expression one still has
  $0\le i_k\le s_k-2$ for $0\le k\le n-1$ but now
  $0\le i_n\le s_n-1=d_{n-1}$.  There is at least one
  multi-index $(I',i_n)$ with $t'_{I',i_n}\ne 0$, otherwise, one would have
  $f_t=f_0$.  Indeed, $\ell^{s_n} = x_n^{s_n} - s_n
  cx_n^{s_n-1}+\dots$. The coefficient of $x_n^{s_n-1}$ in $f_t$ must be
  zero, so $c=0$, $\ell^{s_n}=x_n^{s_n}$ and $f_t=f_0$, which is a
  contradiction. So a   multi-index $(I',i_n)$ with $t'_{I',i_n}\ne 0$ exists.
  Now, in the above notations,
  $\val_\ell (t'_{I',i_n} x^{I'}\ell^{i_n}) = i_n$. One has
  \begin{eqnarray*}
    \frac{\delta_{I'}}{d_{n-1}} + 
    \sum_{k=0}^{n-1} \frac{i_k}{s_k} 
        = 1
    \quad\text{and}\quad
    \frac{i_n}{s_n} + 
    \sum_{k=0}^{n-1} \frac{i_k}{s_k} 
    < 1,
  \end{eqnarray*}
  which implies that
  \begin{displaymath}
    i_n < \delta_{I'} \frac{s_n}{d_{n-1}}
    = \delta_{I'} + 
    \frac{\delta_{I'}}{d_{n-1}} \le \delta_{I'} + 1
  \end{displaymath}
  Thus, $i_n \le \delta_{I'}$ and
  $\ovval_\ell (t'_{I',i_n} x^{I'}\ell^{i_n}) = i_n/\delta_{I'} \le 1$. So
  the reduced valuation at $\ell$ is $\nuu_\ell\le1 $,
  and $X$ is log canonical by
  Theorem~\ref{thm:lct-formula}.  Now assume that $\nuu_\ell=1$, so
  that $X$ is strictly log canonical. Then
  $i_n = \delta_{I'}$ for each $I'$, so we can write simply $t_{I'}$
  instead of $t_{I',i_n}$. Then the polynomial
  $f_t- \ell^{s_n}$ is homogeneous for the degrees
  $\deg x_k = d_{n-1}/s_k = a_{n-1,k}$ for $0\le k\le n-1$ and
  $\deg\ell=1$. Thus, $f_t-\ell^{s_n} = g_{t'}$ for a homogeneous
  degree $d_{n-1}$ polynomial $g_{t'}$ defining a hypersurface in
  $\cym_{n-1}$. 
  Finally, the coefficient of $x_n^{s_n-1}$ in $f_t$
  should be zero, which means that  the part of $f_t$ involving
  only $x_n$, without $x_0,\dotsc, x_{n-1}$, is
  \begin{displaymath}
    \ell^{s_n} + t'_{0,\dotsc,0} \ell^{s_n-1} =
    (x_n- \tfrac1{s_n} t'_{0,\dotsc,0} )^{s_n-1}
    (x_n+ \tfrac{s_n-1}{s_n} t'_{0,\dotsc,0} ),
  \end{displaymath}
  and one must have $\ell= x_n- \tfrac1{s_n}t'_{0,\dotsc,0}$. So $f_t$
  is obtained from $g_{t'}$ by the procedure described in
  Section~\ref{sec:En-1-En}. This completes the proof of
  Theorem~\ref{mainthm:singularities}.
\end{proof}

\subsection{The fiber structure of varieties \texorpdfstring{$[X]\in\cym_n$}{}}
\label{sec:fiber-structure}

Here we show that fiber structure on the open part $X\setminus D$ of a
variety in $\cym_n$ extends to $X$ itself, after a crepant blowup
$\wX\to X$.

\begin{lemma}
  There is a normal toric variety $\wt{\bP}$, a partial crepant
  resolution of singularities $\wt{\bP}\to \amb$ with an irreducible
  exceptional divisor, and a morphism
  $\wt{\bP} \to \bP(1,d_{n-1})=\bP^1$ whose restriction to
  $\bA^1\subset \bP(1,d_{n-1})$ is isomorphic to the projection
  $\bP(a_0,\dotsc, a_{n-1},1)\times\bA^1\to \bA^1$.
\end{lemma}
\begin{proof}
  The fan of $\amb$ lies in the lattice $N\simeq\bZ^{n+1}$ generated by 
  vectors $v_0,\dotsc, v_{n+1}$ with a unique linear relation 
  \begin{displaymath}
    a_{0}v_0 + \dotsb + a_{n-1}v_{n-1} + a_{n}v_n + v_{n+1} = 0.
  \end{displaymath}
  The cones are simplicial cones generated by  proper subsets of
  $\{v_0,\dotsc,v_{n+1}\}$. Let 
  \begin{math}
    \displaystyle
    v_n'  = \frac{a_{n}v_n + v_{n+1}}{s_n}. 
  \end{math}
  Dividing the above relation by $s_n$ we get
  \begin{equation}\label{eq:v-n-prime}
    a_{n-1,0}v_0 + \dotsb + a_{n-1,n-1}v_{n-1} + v_n' = 0,
  \end{equation}
  which shows that $v_n'\in N$. The fan of $\wt{\bP}$ is
  obtained from that of $\amb$ by adding the vector $v_n'$ and then
  subdividing each cone containing both generators $v_n,v_{n+1}$ into
  a pair of cones, with generators $v_n,v_n'$ and $v_n',v_{n+1}$. This
  gives a proper birational morphism $\wt{\bP} \to \amb$, which
  is crepant because $a_{n} + 1 = d_{n-1}+1 = s_n$. The morphism
  $\wt{\bP} \to \bP(1,d_{n-1})$ is defined by the projection
  from $N$ to $N/\la v_0,\dotsc v_{n-1}, v_n'\ra$, which is 
  generated by the images $\bar v_n$ and $\bar v_{n+1}$ with a single
  relation $d_{n-1}\bar v_n + \bar v_{n+1} = 0$, giving the fan of
  $\bP(1,d_{n-1})$. The cones generated by the vectors
  $v_0,\dotsc v_{n-1}, v_n'$ form the fan of $\bP(a_0,\dotsc,a_{n-1},1)$, and the
  cones generated by the vectors $v_0,\dotsc v_{n-1}, v_n', v_n$ form
  the fan of $\bP(a_0,\dotsc,a_{n-1},1)\times \bA^1$, so over $\bA^1$ with the fan
  $\bR_{\ge0}\bar v_n$ the map restricts to the projection
  $\bP(a_0,\dotsc,a_{n-1},1)\times\bA^1\to\bA^1$.
  
\end{proof}

Since the intersection of $X_t$ with the linear subspace
$\{x_n=x_{n+1}=0\} \subset H$ is
transverse, we get:

\begin{corollary}
  For any $X_t\in \cym_n$ there is a crepant blowup $\wX_t \to X_t$
  with a projection $\wX_t\to \bP(1,d_{n-1})$ such that the fibers
  over $\bA^1\subset\bP(1,d_{n-1})$ are either varieties
  $X'_t\in \cym_{n-1}$ or the cuspidal fibers
  $(x_0^{s_0} + \dotsb + x_{n-1}^{s_{n-1}}=0)$, considered as
  a projective variety.  The finite part $X_t\setminus D_t$ is an open
  subset of $\wX_t$ and it is fibered over $\bA^1$ with the fibers
  $X_{t'}\setminus D_{t'} = (g_{t'}=0)$, considered as affine
  varieties.
\end{corollary}

\presectionskip
\section{Two moduli interpretations of the family $\pi_n$}
\label{sec:moduli}

\subsection{As moduli of hypersurfaces}
\label{sec:moduli-hypersurfaces}

Let $A$ be a commutative ring with identity in which $s_k$ for $0\le
k\le n$ are invertible.
We recall that the weighted projective space $\amba$ and its infinite
divisor $\bD_A$ were defined in Section~\ref{sec:wps}. We assume $n>0$.

\begin{lemma}
  The automorphism group of $(\amba, \bD_A)$ fits into a split-exact
  sequence
  \begin{displaymath}
    1\to \bG_m \to \Aut_{\rm graded} A[x_0,\dotsc, x_{n+1}] \to \Aut(\amba,
    \bD_A) \to 1
  \end{displaymath}
  with $\bG_m$ acting on $A[x_0,\dotsc, x_{n+1}]$ by
  $x_k \to \lambda^{a_k} x_k$ for $0\le k\le n$,
  $x_{n+1}\to \lambda x_{n+1}$, and it can be identified with the subgroup
  of $\Aut_{\rm graded} A[x_0,\dotsc, x_{n+1}]$ consisting of graded
  automorphisms satisfying $x_{n+1}\mapsto x_{n+1}$.
\end{lemma}
\begin{proof}
  The automorphism group $\Aut(\amba,\bD_A)$ over $A$ is the subgroup
  of the group $\Aut(\amba,\cO(\bD_A)) = \Aut(\amba,\cO(1))$ sending the
  canonical section $1_{\bD_A} = x_{n+1}$ to itself because by
  Lemma~\ref{lem:wps-U} $i_*\cO_U(1) = \cO_{\amba}(1)$ for the smooth
  locus $U\subset\amba$ and $\bD_A$ is a Cartier divisor on $U$.  By
  Lemma~\ref{lem:wps-U} any $\varphi\in \Aut(\amba, \cO(1))$ sends
  each sheaf $\cO(m)$ to itself, so by Lemma~\ref{lem:wps-cohs} it
  induces a graded automorphism of $R=A[x_0,\dotsc,
  x_{n+1}]$. Conversely, a graded automorphism of $R$ induces an
  automorphism of $(\Proj R, \cO(1))$, and since
  $H^p(\amb_A,\cO(1)) = Ax_{n+1}$ for $n>0$, also an automorphism of
  $(\amba, \bD_A)$. It is clear that the kernel of the surjection
  is~$\bG_m$ acting as stated.
\end{proof}

Any homogeneous degree-$d$ polynomial in $A[x_0,\dotsc, x_{n+1}]$ can be
written in the following form, 
with the extended index set $\cI_n^+$ defined in~\eqref{eq:extended-index-set}:
\begin{equation}\label{eq:Fermat}
  F = \sum_{k=0}^{n} c_k x_k^{s_k}  + 
  \sum_{I\in\cI_n^+} t_I x^I
  \cdot x_{n+1}^{\delta_I}
\end{equation}

\begin{definition}\label{def:fermat-nondegenerate}
  We say that $F$ is \emph{Fermat-nondegenerate} if all $c_k\in A^*$,
  and that $F$ is in
  \emph{normal form} or 
  \emph{short Weierstrass form} if all $c_k=1$ and $t_I=0$
  unless all indices $i_k$ in $I$ satisfy $i_k\le s_k-2$.
\end{definition}

The following result is a higher-dimensional analogue of the short
Weierstrass form for elliptic curves over $\bZ[1/6]$.

\begin{lemma}[Uniqueness of normal form]
  \label{lem:normal-form}
  For any Fermat-nondegenerate polynomial $F$ as in \eqref{eq:Fermat} there exists
  \begin{displaymath}
    \varphi\in \Aut(\amba, \bD_A)\colon \quad
    x_k = p_k(x'_j), \quad x_{n+1} = x_{n+1}'
  \end{displaymath}
  such that $F(x_k) = \Lambda \cdot F'(x_k')$ with $\Lambda\in A^*$
  and $F'(x'_k)$ is in normal form.

  Furthermore, if $F(x_k)$, $F'(x'_k)$ are two polynomials in normal
  form that differ by some $\varphi\in\Aut(\amba, \bD_A)$ and such that
  $F(x_k) = \Lambda\cdot F'(x'_k)$ for some $\Lambda\in A^*$ then
  there exists a unique $\lambda\in A^*$ such that
  \begin{displaymath}
    x_k = \lambda^{-a_k}x'_k
    \text{ for } 0\le k\le n,
    \ x_{n+1}=x'_{n+1} 
    \quad\text{ and }\quad
    t'_I = t_I \cdot \lambda^{\delta_I}
    \text{ for all } I.
  \end{displaymath}

\end{lemma}

\begin{proof}
  For $k=0,\dotsc, n$, replace $x_k$ by
  $x_k-\frac1{s_kc_k}\cdot\text{(the coefficient of $x_k^{s_k-1}$)}$ and
  note that the extra lower-order terms created have exponents
  $\le s_k- 2$, hence are permitted in normal form.
  From here, we find a normal form by a diagonal change of coordinates
  $x_k = u_k x_k'$ for some $u_k\in A^*$. Let us take
  \begin{displaymath}
    \Lambda := \prod_{k=0}^n c_k^{-a_k}
    \quad\text{and}\quad
    u_k := c_k^{-(a_k+1)/s_k}  \prod_{j\ne k} c_j^{-(a_j/s_k)}.
  \end{displaymath}
  The exponents in the expression for $u_k$ are integers by 
  the Congruence relation~\eqref{eq:congruence}.
  It is clear that
  \begin{math}
    u_k^{s_k} = c_k\inv \Lambda, 
  \end{math}
  for all $k$, so
  \begin{math}
    c_k x_k^{s_k} = c_k(u_k^{s_k}x_k'{}^{s_k}) = 
    \Lambda x_k'{}^{s_k},
  \end{math}
  which proves the first part. 

  For the second part, the change of coordinates must be
  diagonal, otherwise it would create a nonzero monomial $t_Ix^I$
  with some $i_k= s_k-1$. So we have $x_k = u_k x_k'$ and $u_k^{s_k} =
  \Lambda$ for all $k$. For each $I$ one has
  \begin{equation}
    \label{eq:tprime-I}
    t'_I = t_I \cdot \Lambda\inv \prod_{k=0}^n u_k^{i_k}.
  \end{equation}
  Since $\gcd(a_k)=1$, there exist integers $b_k$ such that
  $\sum_{k=0}^n a_kb_k = 1$. Define $\lambda:= \prod_{k=0}^n
  u_k^{-b_k}$. For $k\ne j$ one has $a_j/s_k\in\bZ$ and $s_ja_j/s_k
  =a_k$. Thus,
  \begin{displaymath}
    u_k^{a_j} = (u_k^{s_k})^{a_j/s_k} = \Lambda^{a_j/s_k} =
    u_j^{s_ja_j/s_k} = u_j^{a_k}.
  \end{displaymath}
  Therefore,   
  \begin{displaymath}
    \lambda^{-a_j} = \prod_{k=0}^n u_k^{a_j b_k} =
    \prod_{k=0}^n u_j^{a_k b_k} =
    u_j^{\sum a_kb_k} = u_j.
  \end{displaymath}
  In particular, $\Lambda = u_j^{s_j} = \lambda^{-s_ja_j} =
  \lambda^{-d}$. Now,
  \begin{displaymath}
    \lambda^{\delta_I} = \lambda^{d-\sum a_ki_k} =
    \lambda^d \prod_{k=0}^n \lambda^{-a_ki_k} =
    \Lambda\inv \prod_{k=0}^n u_k^{i_k},
  \end{displaymath}
  which together with \eqref{eq:tprime-I} proves the existence of
  $\lambda$. 
  Uniqueness of $\lambda$ is obvious
  from the relations $x_k = \lambda^{-a_k}x'_k$ and $\gcd(a_0,\dotsc, a_n)=1$.  
  (In fact, the $\gcd$ relation~\eqref{eq:linear-relation} implies
  that $\lambda = u_0\inv u_1 \dotsb u_n$.)
\end{proof}

\begin{remark}
  On the affine chart ${x_{n+1}=1}$ the proof works for the Fermat
  nondegenerate hypersurfaces with $0\le i_k\le s_k-1$ even without the
  assumption $\delta_I>0$.
\end{remark}

\begin{proof}[Proof of Theorem~\ref{mainthm:moduli-hypersurfaces}]
  Let $\pi\colon (\cX,\cD)\to S$ be a family which locally can be
  written as a family of Fermat-nondegenerate hypersurfaces given by
  an equation~\eqref{eq:Fermat} with invertible $c_k$. Here,
  ``locally'' could be interpreted to mean in either Zariski, \'etale
  or fppf topology.  We want to show that there exists a unique
  morphism $S\to \cym_n$ such that $\pi$ is the pull-back of the
  universal family $\pi_n$ over~$\cym_n$. A morphism to the weighted
  projective stack $\cym_n= \cP(\delta_I,\, I\in\cI_n)$ is defined by a
  line bundle $L$ on $S$ and the sections $t_I\in  H^0(S,L^{\delta_I})$,
  cf.\ \cite[Lemma 2.1.3]{abramovich2011stable-varieties}. Locally on
  $U_i=\Spec A_i$, the family $\cX\times_S U_i\to U_i$ is defined by
  an equation of the form~\eqref{eq:Fermat}. By
  Lemma~\ref{lem:normal-form}, it can be put into the normal
  form. Choose one such presentation for each~$U_i$. By the uniqueness
  part of Lemma~\ref{lem:normal-form}, on the overlaps
  the presentations differ by
  a uniquely defined element $\lambda_{ij} \in \Gamma(U_{ij}, \cO^*)$. The
  collection $(\lambda_{ij})$ is a $1$-cocycle and it defines an
  element of
  $H^1(S_{\rm Zar}, \bG_m)=H^1(S_{\rm et}, \bG_m)=H^1(S_{\rm fppf}, \bG_m)=\Pic(S)$.
  Thus, it defines an invertible sheaf $L$.
  Clearly,
  $t_I$ are sections of $L^{\delta_I}$ and $\pi$ is the pullback of
  $\pi_n\colon (\cX_n,\cD_n)\to \cym_n$.
\end{proof}

\subsection{As KSBA moduli}
\label{sec:moduli-ksba}

In this section, we work in the category of locally Noetherian schemes
over $\bC$.  We refer to \cite{kollar2023families-of-varieties} for
the KSBA theory. The main definition is this:

\begin{definition}
  A KSBA-stable pair $(X,B)$ consists of a reduced projective variety
  $X$ and a divisor $B=\sum b_iB_i$, with effective $\bZ$-Weil
  divisors $B_i$ and $0<b_i\le 1$ such that $(X,B)$ has slc
  singularities and $K_X+B$ is ample.
\end{definition}

By Theorems~\ref{mainthm:family} and~\ref{mainthm:singularities}, the
pairs $(X,D)$ appearing in the family $\pi_n$ are KSBA stable. In
fact, there is a surprise: assuming $n\ge2$, these pairs are log
canonical rather than merely semi log canonical. The slc condition
allows ordinary double crossing singularities in codimension~$1$, and
typically for KSBA spaces some of the varieties are reducible. But in
the family $\pi_n$ all varieties are normal and irreducible.

Since for every pair $[(X,B)]\in\cym_n$ the dualizing sheaf $\omega_X$
is trivial and $D$ is a $\bQ$-Cartier divisor, most of the
complications which are present in the general case of the KSBA theory
disappear and we can restrict to the simplest version of the
definition for a family of stable pairs:

\begin{definition}\label{def:ksba-family}
  We define $\cM$ to be the stack of families
  $\pi\colon (\cX,\cD)\to S$ over Noetherian schemes $S$ such that
  both $\cX\to S$ and $\cD\to S$ are flat families of varieties
  and $\cD\subset\cX$ is a reduced effective $\bZ$-divisor such that:
  \begin{enumerate}
  \item $d\cD$ is Cartier and relatively ample.
  \item $\omega_{\cX/S}$ is $\pi$-trivial, i.e.\ it equals $\pi^*(L)$
    for some line bundle $L$ on $S$.
  \item Every geometric fiber $(X,D)$ is KSBA stable.
  \item For every $m\ge 0$ the sheaves
    $\omega_{\cX/S}^m(m\cD)$, equivalently $\cO_{\cX/S}(m\cD)$, are flat over $S$ and for
    every $s\in S$ one has
  \begin{math}
    \cO_\cX(m\cD)|_{\cX_s} = \cO_{\cX_s}(m\cD_s).
  \end{math}
  \end{enumerate}
\end{definition}

\begin{remark}
  This definition is a version of ``Alexeev stability''
  \cite[Section 6.4]{kollar2023families-of-varieties}.
  Weakening this definition in some natural way, for example removing
  conditions (1) or (2) and compensating that by giving more technical
  conditions, merely leads to an open substack of $\cM$. Also, over a
  reduced base $S$, (4) follows from (1,2,3) by \cite[Corollary
  4.33]{kollar2023families-of-varieties}.
\end{remark}

\begin{proof}[Proof of Theorem~\ref{mainthm:moduli-ksba}]
  We first note that as a combination of the definition of the family
  $\pi_n$ in Section~\ref{sec:def-of-family} and
  Lemmas~\ref{lem:basic-family-canclass},~\ref{lem:cO_X(m)-are-flat},
  $\pi_n$ \emph{is} a family of KSBA-stable pairs of
  Definition~\ref{def:ksba-family}.  
  Now let $\pi\colon (\cX,\cD)\to S$ be any family as in
  Definition~\ref{def:ksba-family}, and $s\in S$ be a point such
  that the fiber $(X,D)= (\cX_s,\cD_s)$ is in $\cym_n$.  By
  Lemma~\ref{lem:cohs-cX}, for any $m\ge0$ the cohomology groups
  $H^i(X, \cO(mD))$ vanish for $i>0$. By semicontinuity
  \cite[Theorem~12.8]{hartshorne1977algebraic-geometry}, this is also
  true for any $s'$ in an open neighborhood $U\ni s$. Then by the
  Cohomology and Base Change
  \cite[Theorem~12.11]{hartshorne1977algebraic-geometry} the
  $\cO_U$-algebra
  \begin{displaymath}
    R(\cX_U, \cD_U) = \oplus_{m\ge 0}\, \pi_* \cO_{\cX_U}(m\cD_U)
  \end{displaymath}
  is locally free. By shrinking $U$, we can assume that it is a free
  $\cO_U$-algebra.  For the central fiber $(X,D)$ we can choose 
  generators $x_k$ of $R(X, D)$ of degree $a_k$ for $0\le k\le n$ and
  degree~$1$ for $k=n+1$, and after shrinking $U$
  we can extend them to generators of $R(\cX_U,\cD_U)$.
  The choice of $x_k$ gives a homomorphism of graded $\cO_U$-algebras
  \begin{math}
    \Sym^*(V) \to R(\cX_U, \cD_U)
  \end{math}
  where $V=\oplus_{k=0}^{n+1} V_k$ is defined as in
  Section~\ref{sec:def-of-family} with $L=\cO_U$. This homomorphism is
  surjective when restricted to $s\in U$, so by Nakayama's lemma it is
  surjective over $U$ after shrinking it, with a kernel
  $\cO_U \cdot F$.  This gives a closed embedding
  $(\cX_U,\cD_U) \subset \bP(V)$ and shows that $\cX_U\to U$ is a
  family of hypersurfaces $(F=0)$. Since the central fiber is Fermat
  nondegenerate, the same is true over $U$, after possibly shrinking
  it further. A different choice of $x_k$ defines an automorphism of
  $V$ and automorphism of $(\amb_U,\bD_U)$. By
  Lemma~\ref{lem:normal-form} we can put it into a normal 
  form.  So $(\cX_U,\cD_U)\to U$ is a family in $\cym_n$.

  This proves that the functor of KSBA-stable pairs $(X,D)$ is open in
  $\cym_n$. Since the entire family $\pi_n$ is a family of KSBA-stable
  pairs, it follows that $\cym_n$ is a connected component of the
  moduli stack of KSBA-stable pairs. 
\end{proof}

\presectionskip
\section{Hodge theory of Calabi-Yau varieties 
  \texorpdfstring{$X$}{X} in \texorpdfstring{$\cym_n$}{E_n}}

\label{sec:hodge-theory}

\subsection{Results of Esser-Totaro-Wang and Singh} 
\label{sec:known-hodge}

The Calabi-Yau hypersurfaces $X\subset \amb$, $[X]\in\cym_n$, have
appeared in Esser-Totaro-Wang \cite{esser2023varieties-of-general,
  esser2022calabi-yau} and Singh \cite{singh2025smooth-calabi}, where
they are denoted by $X^{(n)}_2 \subset \bP_2^{(n+1)}$. We list several
facts about these varieties proved in these papers.

\begin{theorem}[\cite{esser2022calabi-yau, singh2025smooth-calabi}]
  \label{thm:self-dual}
  The Fano polytope of $\amb$ is linearly isomorphic, as a
  lattice polytope, to its polar dual. Thus, the family of Calabi-Yau
  hypersurfaces $X\in |-K_{\amb}|$ is self mirror-dual.
\end{theorem}

\begin{remark}
  This generalizes to dimensions $n\ge3$ the fact that the triangle
  singularity $D_{2,3,7}$ is self-dual for Arnold's strange duality
  and that the moduli space of $U\oplus E_8$-polarized K3 surfaces is
  self mirror-dual.
\end{remark}

\begin{theorem}[\cite{singh2025smooth-calabi}]
  \label{thm:crepant-resolution}
  The weighted projective space $\amb$ admits a projective
  crepant resolution of singularities. For a generic hypersurface
  $X\in |-K_{\amb}|$ it induces a projective crepant resolution of
  singularities $\phi\colon \wX\to X$.
\end{theorem}

The paper \cite{esser2022calabi-yau} proves several facts about the
orbifold cohomology groups of a generic Calabi-Yau variety
$X\in\cym_n$.  By \cite[Corollary 1.5]{yasuda2004twisted-jets} they
coincide with the ordinary cohomology groups of a crepant resolution
$\wX$.

\begin{theorem}[{{\cite[Theorem 5.1]{esser2022calabi-yau}}}]
  \label{thm:hpq-etw}
  Let $[X]\in\cym_n$.  Then
  \begin{enumerate}
  \item $h^{p,q}\orb(X) = 0$ unless $p=q$ or $p+q=n$.
  \item The sum of the orbifold Betti numbers of $X$ is $2\mu$. 
  \item If $n$ is odd then $\dim H^n\orb(X, \bQ) = \mu$.
  \end{enumerate}
\end{theorem}

Recall that $\mu$ is the Milnor number of the singularity
\eqref{eq:brieskorn-pham}.

\subsection{Hodge numbers}
\label{sec:hodge-numbers}

Recall the definition of $N_p$ from~\eqref{eq:N_p}. Then
$\sum_{p=0}^n N_p = \mu$.
Let $\dotcym_n\subset \cym_n$ be the open subset over which the varieties
$X\setminus D$ are smooth.

\begin{theorem}[Theorem~\ref{mainthm:hodge-numbers}]
  \label{thm:hodge-numbers}
  For $[X]\in\dotcym_n$, one has $h^{p,p}\orb(X) =
  h^{p,n-p}\orb(X) = N_p$ with one exception: if $n=2p$ is even then
  $h^{p,p}\orb(X) = 2N_{p}$. Other $h^{p,q}\orb(X)$ vanish.
\end{theorem}
\begin{proof}
  As in Section~\ref{sec:sylvester-numbers}, let $a_k = d/s_k$ for
  $0\le k\le n$. Also, set $a_{n+1}=1$.
  By \cite[\S5]{esser2022calabi-yau}, $h^{p,q}\orb(X)$ is the coefficient
  of $(x^p y^q)^d$ in 
  $Q(x,y) = \sum_{\ell=0}^{d-1} Q_\ell(x,y)$, where
\[
  Q_\ell(x,y) =
  \prod_{\substack{0\leq k\le n+1\\\hskip 8pt\wt\theta_k(\ell) =0}}
  \Biggl(
  \sum_{i_k=0}^{d/a_k - 2} (xy)^{a_k i_k} 
  \Biggr) 
  \prod_{\substack{0\le k\le n+1\\\hskip 8pt\wt\theta_k(\ell) \ne 0}}
  (xy)^{d/2-a_k}\,
  \Bigl(\frac{x}{y}\Bigr)^{d\{ \ell a_k/d\}-d/2}
\]
In this formula, $\wt\theta_k(\ell)=\{\ell a_k/d\}$. For
$0\le k\le n$ the condition $\wt\theta_k(\ell)=0$ means that $s_k | \ell$,
and for $k=n+1$ that $\ell=0$.  There are two natural
sets of size $\mu$:
\begin{enumerate}
\item The $(n+1)$-tuples $(i_0,\dotsc, i_n)$ of integers with $0\le i_k \le
  s_k-2$, 
\item integers $\ell$ with $0\le \ell \le d-1$ such that
  $s_k \nmid \ell$  for $0\le k\le n$.
\end{enumerate}
and a natural bijective correspondence between them:
\begin{displaymath}
  (i_0,\dotsc, i_n) \quad\longleftrightarrow\quad
  \ell \text{ such that } \ell \equiv i_k + 1 \mymod{s_k}.
\end{displaymath}
Indeed, the condition $0\le i_k\le s_k-2$ is equivalent to
$i_k+1 \not\equiv 0 \mymod{s_k}$.
Now, for any collection $(i_0,\dotsc, i_n)$ in (1), there are
unique $0\le p\le n$ and $0\le i_{n+1} \le d-1$ that appear in the
identity
\begin{equation}\label{eq:Q0}
  \sum_{k=0}^n a_ki_k + i_{n+1} = pd
  \quad\iff\quad
  \sum_{k=0}^n \frac{i_k}{s_k} + \frac{i_{n+1}}{d} = p.
\end{equation}
We compute:
\begin{displaymath}
  i_{n+1}\mymod{s_k} = -a_ki_k\mymod{s_k} = i_k\mymod{s_k}
\end{displaymath}
by the Congruence relation~\eqref{eq:congruence},
which implies that $\ell = i_{n+1} + 1$. In particular,
$i_{n+1}\ne d-1$, so in fact one has $0\le i_{n+1}\le d-2 =
d/a_{n+1}-2$, as in $Q_\ell(x,y)$. For a fixed $p$,
the solutions of \eqref{eq:Q0} are in bijection with those of
\eqref{eq:N_p}, so there are $N_p$ of them.
For each element of these two sets there is a monomial in
$Q(x,y)$:

\smallskip (1) In $Q_0(x,y)$ all $\wt\theta_k(0)=0$, and a tuple
$(i_0,\dotsc, i_n)$ satisfying \eqref{eq:Q0} gives the monomial
$(x^py^p)^d$, contributing to $h^{p,p}$.

\smallskip (2) For the $\ell$ corresponding to $(i_0,\dotsc, i_n)$
from (1), all
$\wt\theta_k(\ell) \ne 0$, so $Q_\ell(x,y)$ is a single monomial
$x^A y^B$, where, using the Egyptian fraction identity~\eqref{eq:egyptian}
and \eqref{eq:Q0}, 
\begin{eqnarray*}
  A + B &=& 2\Bigl( (n+2)\frac{d}{2} - \sum_{k=0}^{n+1} a_k \Bigr) = nd
            \qquad \text{and}\\
  A &=& -\sum_{k=0}^{n+1} a_k +  \sum_{k=0}^n
  d\Bigl\{ \frac{\ell}{s_k}\Bigr\} + d\Bigl\{ \frac{\ell}{d} \Bigr\} =
  d \Bigl( -1 + \sum_{k=0}^n
    \Bigl\{ \frac{\ell}{s_k}\Bigr\} + \frac{\ell}{d}
        \Bigr) \\
  &=&  d \Bigl( -1 + \sum_{k=0}^n \frac{i_k+1}{s_k} +
      \frac{i_{n+1}+1}{d} \Bigr)
 = d(-1 + p + 1) = dp.
\end{eqnarray*}
Thus, $Q_\ell(x,y) = (x^p y^{n-p})^d$, contributing to $h^{p,n-p}$.

To summarize: we have found two sets of monomials in $Q(x,y)$,
each of size $\mu$, that
contribute to the Betti numbers. By Theorem~\ref{thm:hpq-etw}(2),
the sum of the Betti numbers is $2\mu$, so we have found them all.
Thus, the only nonzero $h^{p,q}$ lie on one of the main diagonals
$h^{p,p}$ and $h^{p,n-p}$, each diagonal contributing the values
$(N_0,N_1,\dotsc, N_n)$. If $n$ is even, these diagonals intersect at
the center, doubling the entry $h^{n/2,n/2}$.
\end{proof}

\begin{corollary}\label{cor:h11}
  For $n\ge 3$ one has $h^{1,1}\orb = h^{n-1,1}\orb = N_1 = \dim\cym_n$.
\end{corollary}

Recall that for a crepant resolution of singularities $\wX\to X$, if
one exists, one has
$h^{n-1,1}\orb(X) = h^{n-1,1}(\wX) = \dim H^1(T_\wX)$, and since
deformations of smooth Calabi-Yau varieties are unobstructed by
Bogomolov-Tian-Todorov, this is the dimension of the space of local
deformations of $\wX$.

\medskip

Because the Sylvester numbers grow doubly-exponentially, for large $n$ it
becomes impractical to compute the exact values of $N_p$. Instead, we could
 find asymptotics for $N_p/\mu$ by considering a sum
$S_n = \sum_{k=0}^n \xi_k$ of $n+1$ independent variables, with $\xi_k$
uniformly distributed on the set
$\{0, \frac1{s_k}, \frac2{s_k}, \dotsc \frac{s_k-2}{s_k}\}$. This sum
takes values in the interval $[0, n-1+\frac2{d_n}]$, which is
extremely close to $[0,n-1]$. The expectation and the variance of $S_n$ are
\begin{displaymath}
  \bE(S_n) = \frac{n-1}2 + \frac1{d_n}, \quad
  \operatorname{Var}(S_n) = \frac{n-1}{12} + \frac1{6d_n}, 
\end{displaymath}
which match the moments of the uniform sum distribution
$\sum_{k=1}^{n-1} U_k$ of $n-1$ random variables
$U_k = \operatorname{Unif}(0,1)$, the Irwin-Hall distribution of order
$n-1$, except for a vanishingly small error on the order of
$1/d_n$. This gives an estimate
  \begin{displaymath}
    \displaystyle
    \frac{N_p}{\mu} \approx \frac{A(n-1,p-1)}{(n-1)!},
    \quad\text{where}\quad
    A(n-1,p-1) = \sum_{j=0}^p (-1)^j \binom{n}{j} (p-j)^{n-1}.
  \end{displaymath}
  Numerical checks show that this approximation works exceedingly
  well. In particular, one has
\begin{equation}\label{eq:dim-cym}
  \displaystyle\dim\cym_n = N_1 \approx \frac{\mu}{(n-1)!}.
\end{equation}

\subsection{Period map and local Torelli theorem} 
\label{sec:period-map}

Let $\dotcym_n\subset \cym_n$ be the open subset over which the varieties
$X\setminus D$ are smooth.
Let $[X]\in\dotcym_n$ and $\wX$ be its
crepant resolution. Let $H= H^n\orb(X) = H^n(\wX)$ be the middle
cohomology group, together with the cup product pairing
$H\times H\to \bZ$, which is symmetric for even $n$ and skew-symmetric
for odd $n$. The period domain $\fD$ for the Calabi-Yau varieties in
$\cym_n$ is the parameter space of flags
\begin{displaymath}
  H_\bC = F^0 \supset F^1 \supset \dotsb \supset F^n \supset 0
\end{displaymath}
with the graded pieces
$\operatorname{Gr}^p H^n = F^p/F^{p+1} \simeq H^{p,n-p}$ and
$H^{p,q} = F^p \cap \overline{F^q}$, which satisfy the Hodge
identities $H^{q,p} = \overline{H^{p,q}}$.  The family
$\dot\pi_n\colon \dotcX^{(n)}\to \dotcym_n$ defines the monodromy group
$\Gamma \subset \Aut(H, \la\cdot,\cdot\ra)$ and a period map
$\Phi\colon \dotcym_n \to \fD/\Gamma$.

\begin{theorem}[Theorem~\ref{mainthm:kodaira-spencer}]
  \label{thm:kodaira-spencer}
  For $n\ge3$, at a point $[X]\in\dotcym_n$, the Kodaira-Spencer map
  $\rho\colon T_{[X]} \cym_n \to H^1\orb(T_X)=H^1(T_\wX)$ is an isomorphism.
\end{theorem}
\begin{proof}
  The orbifold cohomology of $X$ can be equivalently interpreted in
  two ways: as the cohomology of $X$ considered to be a smooth DM stack,
  or as the cohomology of a crepant resolution of singularities $\wX$
  of $X$. We take the first point of view and consider both $X$ and
  $\amb$ as smooth DM stacks.  Then the usual argument for smooth
  Calabi-Yau hypersurfaces in an ordinary projective space applies
  verbatim.  Below all sheaves and cohomology groups are on
  smooth DM stacks. 

  Let $\Def(X)$ be the space of deformations of $X$
  as a hypersurface in $\amb$. One has a canonical identification
  $T_{[X]} \Def(X) = H^0(X, N_{X/\amb})$. From the exact sequence
  \begin{displaymath}
    0 \to T_X \to T_{\amb}|_X \to N_{X/\amb} \to 0
  \end{displaymath}
  one concludes that the Kodaira-Spencer map
  $T_{[X]} \Def(X) \to H^1(T_X)$ is the connecting homomorphism
  between the cohomology groups from this sequence. Using Euler's
  exact sequence
  \begin{displaymath}
    0\to \cO \to \oplus_{k=0}^{n+1} \cO(a_k) \to
    T_{\amb} \to 0
    \quad \text{with } a_{n+1} =1
  \end{displaymath}
  and the standard vanishing theorems for the sheaves $\cO_{\amb}(d)$, and
  tracing the long exact sequences of cohomologies, for $n\ge3$ we
  obtain a short exact sequence:
  \begin{displaymath}
    H^0(T_{\amb}) \to H^0(N_{X/\amb}) \to H^1(T_X) \to 0
  \end{displaymath}
  Thus, the Kodaira-Spencer map for $\Def(X)$ is surjective, and the
  kernel is the image of $H^0(T_{\amb})$, which is the group of
  infinitesimal automorphisms of $\amb$.  As in the proof of
  Lemma~\ref{lem:normal-form}, the group $\Aut(\amb)$ is
  exactly the group that puts an equation $F_t$ of a Calabi-Yau
  hypersurface $X = (F_t=0) \subset\amb$ into the standard
  form~\eqref{eq:F_t}. Thus,
  $$T_{[X]}\cym_n = T_{[X]}\Def(X) / \im H^0(T_{\amb}).$$ This
  shows that
  the map $\rho$
  for $\cym_n$ at the point $[X]$ is an isomorphism.
\end{proof}

\begin{corollary}[Local Torelli theorem] \label{cor:local-torelli}
  For $n\ge3$, at a point $[X]\in\dotcym_n$
  the period map $\Phi\colon \dotcym_n \to \fD/\Gamma$ is an immersion.
\end{corollary}
\begin{proof}
  This is well known to follow from
  Theorem~\ref{thm:kodaira-spencer} by Griffiths
  \cite[Prop. 2.16]{griffiths1970periods-of-integrals} (see also
  \cite[Thm. 10.4]{voisin2007hodge-theory}) which says that the
  differential $d\Phi$ of the period map is the composition of the
  Kodaira-Spencer map with the map
  \begin{displaymath}
    \kappa\colon H^1(T_\wX) \longrightarrow
    \oplus_{p=0}^n \Hom\bigl(H^{n-p,p}(\wX), H^{n-p-1,p+1}(\wX)\bigr),
  \end{displaymath}
  given by the cup-product and the interior product, and for a
  Calabi-Yau variety the $p=0$ component of $\kappa$ is an
  isomorphism.
\end{proof}

\begin{question}
  Does the global Torelli theorem hold for the family
  $\dotcX^{(n)}\to \dotcym_n$? In other words, is the period map an embedding?
\end{question}

\begin{remark}
  For $n=2$, $\dotcym_2$ is the moduli space of smooth $U\oplus E_8$-polarized
  K3 surfaces and has dimension~$10$, while $\dim H^1(T_X)=20$. The
  global Torelli theorem does hold, if the period domain is taken to be the
  period domain of $U\oplus E_8$-polarized K3 surfaces. 
\end{remark}

\bibliographystyle{amsalpha}

\def\cprime{$'$}
\providecommand{\bysame}{\leavevmode\hbox to3em{\hrulefill}\thinspace}
\providecommand{\MR}{\relax\ifhmode\unskip\space\fi MR }
\providecommand{\MRhref}[2]{%
  \href{http://www.ams.org/mathscinet-getitem?mr=#1}{#2}
}
\providecommand{\href}[2]{#2}

\end{document}